\documentclass{mod}
\usepackage{url}
\usepackage{setspace}
\usepackage{scrextend}
\usepackage{cite}
\usepackage[shortlabels]{enumitem}
\usepackage{hyperref}
\usepackage{caption}
\usepackage{subcaption}
\usepackage{titletoc}
\usepackage{amsmath}
\usepackage{amsthm}
\usepackage{amssymb}
\DeclareMathAlphabet{\mathbbm}{U}{bbm}{m}{n}

\usepackage{mathtools}
\usepackage{mathrsfs}

\usepackage[all]{xy}
\usepackage[toc,page]{appendix}
\usepackage{titlesec}
\usepackage{upgreek}

\usepackage{tocloft}

\usepackage{tocbasic}
\DeclareTOCStyleEntry[
beforeskip=.2em plus 1pt,
entryformat=\normalfont ,
pagenumberformat=\bfseries
]{tocline}{section}

\usepackage{ stmaryrd }

\newcounter{mainthm}

\newcounter{mainconj}

\newtheorem{thm}{Theorem}[section]

\theoremstyle{plain}

\newtheorem{prop}[thm]{Proposition}

\newtheorem{defn-thm}[thm]{Definition-Theorem}

\newtheorem{defn-lem}[thm]{Definition-Lemma}

\newtheorem{conjecture}
[thm]{Conjecture}
\newtheorem{question}
[thm]{Question}

\theoremstyle{definition}
\newtheorem{defn}[thm]{Definition}

\newtheoremstyle{rmk}
{5pt}
{5pt}
{}
{}
{\itshape}
{}
{.5em}
{}

\theoremstyle{rmk}
\newtheorem{rmk}[thm]{Remark}

\newtheoremstyle{note}
{8pt}
{5pt}
{\itshape}
{10pt}
{\bfseries}
{}
{.5em}
{}

\theoremstyle{note}

\setlist[description]{font=
	\normalfont
	\itshape
	\space}

\DeclareMathOperator{\dist}{dist}

\DeclareMathOperator{\Spec}{Spec}
\DeclareMathOperator{\trop}{\mathfrak{trop}}

\DeclareMathOperator{\val}{\mathsf{v}}

\DeclareMathOperator{\median}{\mathrm{median}}

\newcommand*{\Scale}[2][4]{\scalebox{#1}{$#2$}}%

\usepackage{tikz-cd}

\titleformat{\subsection}[runin]{
	\bfseries\itshape\normalsize}{\thesubsection \ }{0em}{}[\mbox{ . } ]
\titlespacing{\subsection}
{0pt}
{2.25ex plus 1ex minus .2ex}
{0pt}

\titleformat{\subsubsection}[runin]{
	\itshape\normalsize}{\thesubsubsection \ }{0em}{}[\mbox{} ]
\titlespacing{\subsubsection}
{0pt}
{2.25ex plus 1ex minus .2ex}
{0pt}

\titlecontents{section}
[0em]
{ \footnotesize \vspace{0pt}}%
{ \thecontentslabel.\enspace}
{}
{\titlerule*[0.5pc]{.}\contentspage}%

\begin{document}
	\setlength{\parindent}{15pt}	\setlength{\parskip}{0em}
	
	\title{Family Floer SYZ singularities for the conifold transition}
	\author[Hang Yuan]{Hang Yuan}
\begin{abstract} {\sc Abstract:}  
We show a mathematically precise version of the SYZ conjecture, proposed in the family Floer context, for the conifold with a conjectural mirror relation between smoothing and crepant resolution.
The singular T-duality fibers are explicitly written and exactly correspond to the codimension-2 `missing points' in the mirror cluster variety, which confirms the speculation of Chan, Pomerleano, and Ueda but only in the non-archimedean setting.
Concerning purely the area of Berkovich geometry and forgetting all the mirror symmetry background, our B-side analytic fibration is also a new explicit example of Kontsevich-Soibelman's affinoid torus fibration with singular extension.
\end{abstract}

	\maketitle




\tableofcontents

%
%

\section{Introduction}

The geometric understanding of mirror symmetry for Calabi-Yau manifolds hinges on the Strominger-Yau-Zaslow (SYZ) conjecture \cite{SYZ}.  This framework has been extended to encompass non-Calabi-Yau settings (c.f. Hori \cite{hori2002mirror}, Auroux \cite{AuTDual}). At its core, the SYZ proposal posits the existence of special Lagrangian torus fibrations on mirror space pairs.  However, the incorporation of quantum corrections to account for wall-crossing phenomena adds complexity to the duality picture.  Indeed, the precise mathematical formulation of the SYZ conjecture, particularly regarding singularities, remained elusive for a significant time due to the limitations of the underlying physics-based notion of T-duality. Thus, the application of the SYZ philosophy and T-duality viewpoint often remained at a heuristic level. Despite these limitations, it is noteworthy that this approach has exhibited surprising connections to successful proofs of Kontsevich's Homological Mirror Symmetry (HMS) \cite{KonICM} for specific examples.

This paper delves deeper into the geometric aspects of SYZ mirror symmetry for the conifold by establishing a mathematically precise formulation of T-duality.  Our work provides a complementary perspective to the categorical HMS results achieved by Chan-Pomerleano-Ueda \cite{Chan_Pomerleano_Ueda} within the SYZ framework. We propose that a quantum-corrected family Floer functor \cite{AboFamilyFaithful} could offer a more conceptually unified explanation.
Our analysis also suggests a noteworthy connection between codimension-two missing points in the mirror cluster variety and the dual SYZ singular fibers. This finding strengthens the speculation of Chan-Pomerleano-Ueda \cite[\S 2.5]{Chan_Pomerleano_Ueda} and emphasizes the potential value of investigating non-archimedean analytification (Section \ref{sss_speculation_Chan}). Our T-duality construction has the potential to shed light on the conjectured mirror symmetry relationship between smoothing and crepant resolution from a more geometric point of view.


\subsection{Main result}
By the \textit{conifold}, we mean the singular algebraic variety
\[
Z=Z(\Bbbk)=\{(u_1,v_1, u_2,v_2)\in\mathbb \Bbbk^4\mid u_1v_1=u_2v_2\}
\]
over a field $\Bbbk$ where we may choose $\Bbbk$ to be either $\mathbb C$ or the Novikov field
$
\Lambda=\mathbb C((T^{\mathbb R})) 
$
on the A-side or the B-side respectively.
The latter $\Lambda$ is a non-archimedean field with a natural valuation map $\val: \Lambda\to\mathbb R\cup \{+\infty\}$.
We denote the projective line over $\Lambda$ by
$
\mathbb P_\Lambda =\Lambda \cup \{\infty\}
$.
Abusing the notation, the above valuation map $\val$ has a natural extension
$
\val: \mathbb P_\Lambda \to  \overline{\mathbb R} \equiv \mathbb R\cup \{\pm \infty\}
$
such that $\val(\infty)=-\infty$. Note that we may think of $\mathbb P_\Lambda$ as the union of two affine lines.

On the A-side, we work over $\mathbb C$, and a smoothing of $Z=Z(\mathbb C)$ is given by
$
X'=\{(u_1,v_1, u_2,v_2)\in\mathbb C^4\mid u_1v_1-c_1=u_2v_2-c_2 \}
$
for some fixed small numbers $c_1>c_2>0$. Consider the anti-canonical divisor 
$
\mathscr D=\{u_1v_1-c_1=0\}=\{u_2v_2-c_2=0\}
$, and its complement in $X'$ is
\begin{equation}
	\label{X_mirror_alg_var_eq_intro}
X=\{(u_1 ,v_1,u_2,v_2, z) \in \mathbb C^4\times \mathbb C^* \mid u_1v_1-c_1=u_2v_2-c_2=z \}
\end{equation}

On the B-side, we consider the same variety $Z=Z(\Lambda)$ but over $\Lambda$.
One of its crepant resolutions is the algebraic variety $Y'=\mathcal O_{\mathbb {P}_\Lambda}(-1)  \oplus  \mathcal O_{\mathbb {P}_\Lambda}(-1)$ 
that consists of the tuple $(z, u_1,v_1, u_2, v_2)$ in $\mathbb P_\Lambda\times \Lambda^4$ such that $u_1z=v_1$ and $u_2z=v_2$.
Removing the divisor $\mathscr E=\{ (v_1-1)(u_2-1)=0\}$
yields the following algebraic variety:
\begin{equation}
	\label{Y_mirror_alg_var_eq_intro}
	Y= \left\{ (x_1, x_2, z, y_1,y_2)\in\Lambda^2\times \mathbb P_\Lambda \times (\Lambda^*)^2  \ \  \mid \ \  
	\begin{aligned}
		x_1 z &=1+y_1 \\
		x_2 &=(1+y_2) z
	\end{aligned}
 \, \,
	\right\}
\end{equation}

Now, the main result of this paper is as follows:

\begin{thm}
	\label{Main_theorem_SYZ_intro}
	$Y$ is SYZ mirror to $X$.
\end{thm}


\begin{defn}
	In the situation of Conjecture \ref{conjecture_SYZ} below, if the conditions (i) (ii) (iii) hold and the analytic space $\mathscr Y$ embeds into (the analytification $Y^{\mathrm{an}}$ of) an algebraic variety $Y$ over $\Lambda=\mathbb C((T^{\mathbb R}))$ of the same dimension, then we say $Y$ is {\textit{SYZ mirror}} to $X$.
\end{defn}


\begin{conjecture}
	\label{conjecture_SYZ}
	Given any Calabi-Yau manifold $X$,
	
	\begin{enumerate}[(a)]
		\itemsep 0pt
		\item  there exists a Lagrangian fibration $\pi:X\to B$ onto a topological manifold $B$ such that the $\pi$-fibers are graded with respect to a holomorphic volume form $\Omega$;
		\item  there exists a tropically continuous map $f : \mathscr Y \to B$ from an analytic space $\mathscr Y$ over the Novikov field $\Lambda=\mathbb C((T^{\mathbb R}))$ onto the same base $B$;
	\end{enumerate}

satisfying the following

	\begin{enumerate}[(i)]
			\itemsep 0pt
		\item  $\pi$ and $f$ have the same singular locus skeleton $\Delta$ in $B$;
		\item  $\pi_0=\pi|_{B_0}$ and $f_0= f|_{B_0}$ induce the same integral affine structures on $B_0=B\setminus \Delta$;
		\item $f_0$ is isomorphic to the canonical dual affinoid torus fibration $\pi_0^\vee$ associated to $\pi_0$.
	\end{enumerate}
\end{conjecture}

A more precise statement would describe $(\mathscr Y, f)$ as SYZ mirror to $(X,\pi)$. For clarity, we use the simplified statement as above. Note that we mainly focus on the \textit{fibration} $f$ instead of the \textit{space} $\mathscr Y$.
The SYZ conjecture focuses on finding mirror fibration partners instead of merely identifying mirror spaces. 
Although the mirror space identification between $X$ and $Y$ for the conifold is known, instances of fibration duality with singular fibers are exceptionally rare in the literature of SYZ conjecture.


\[
\xymatrix{
	X \ar[d]_{\pi} & X_0\ar[l] \ar[d]_{\pi_0}   \ar@{.}@/_1.15pc/[rr]_{\small \text{` T-duality '}}    & & \mathscr Y_0\ar[d]^{f_0} \ar[r]  & \mathscr Y \ar[d]^f \\
	B & B_0 \ar[l] \ar@{.}@/^1.15pc/[rr] & &  B_0 \ar[r] & B
}
\]

\subsection{Outline of the story}
\label{ss_outline_intro}

To comprehend the SYZ fibration duality, we usually start with a fibration denoted as $\pi: X\to B$, where the general fiber is represented by a Lagrangian torus, and the discriminant locus is expressed as $\Delta\subset B$ with $B_0=B-\Delta$. Initially, the "dual" of $\pi_0=\pi|_{B_0}$ can be depicted as the corresponding dual torus fibration $f_0: \mathscr Y_0 \cong R^1\pi_{0*}(U(1))\to B_0$. A genuine "dual" of $\pi$ should be a compactification or an extension of $f_0$ to a particular $f$ (see Gross's introduction in \cite{Gross_topo_MS}).

The Lagrangian torus fibration $\pi_0$ inherently defines an integral affine structure on $B_0$, so a suitable dual fibration $f_0$ should preserve this structure. Locally, $\pi_0$ is based on the logarithm map, $\mathrm{Log}: (\mathbb C^*)^n\to\mathbb R^n$, $z_i\mapsto \log |z_i|$. By choosing an appropriate atlas $(U_i\to V_i)$ for the integral affine structure, we can view $\pi_0:X_0\to B_0$ as a gluing of local integrable system segments $\mathrm{Log}^{-1}(V_i)\to V_i$. In particular, we may view $X_0$ as 
\[
\coprod \mathrm{Log}^{-1}(V_i) / \sim
\]
where $\sim$ indicates a gluing in the category of symplectic manifolds.

On the mirror side, a dual fibration $f_0$ is expected to incorporate quantum correction data for $\pi_0$ via holomorphic disks (cf. \cite{Joyce_book}, Section 9.4.1).  Since quantum corrections involve these disks, Floer theory becomes a natural framework.  However, Gromov's compactness theorem, a cornerstone of symplectic geometry, guarantees Floer-theoretic invariant convergence only over the non-archimedean Novikov field $\Lambda = \mathbb{C}((T^\mathbb{R}))$ instead of the complex numbers.  This naturally leads us to the non-archimedean version of integrable systems developed by Kontsevich and Soibelman \cite{KSAffine}, nowadays called \textit{affinoid torus fibration} \cite{NA_nonarchimedean_SYZ}.
The tropicalization map, $\trop: (\Lambda^*)^n \to \mathbb{R}^n$, offers a non-archimedean counterpart to the logarithm map, $\mathrm{Log}$.  
A key distinction lies in the substitution of the archimedean norm in $\mathbb{C}$ with its non-archimedean counterpart in $\Lambda$. We are then led to a natural question: can a mirror fibration $f_0: \mathscr{Y}_0 \to B_0$ exist that retains a similar composition of local segments $\trop^{-1}(V_i) \to V_i$, but substitutes $\trop$ for $\mathrm{Log}$? 
\[
\Scale[0.8]{
\xymatrix
{
X_0 \ar[dr]|-{\pi_0} & \mathrm{Log}^{-1}(V_i) \ar[dr] \ar@{.}[l] & & \trop^{-1}(V_i) \ar[dl] \ar@{.}[r]& \textbf ? \ar@{-->}[dl]|-{\textbf ?}    \\
& B_0 & V_i	\ar@{.}[l] \ar@{.}[r] & B_0
}}
\]
If so, the resulting $f_0:\mathscr Y_0\to B_0$ would be an affinoid torus fibration.
The total space of such a affinoid torus fibration looks like
\[
\coprod \trop^{-1}(V_i) / \sim
\]
where $\sim$ refers to a gluing in the category of non-archimedean analytic spaces yet. 
Indeed, if the $\{V_i\}$ originate from the same integral affine structure \textit{atlas} on $B_0$, then the above decomposition automatically preserves the structure (cf. \cite[\S 4.1]{KSAffine}). This makes it a strong candidate for the dual fibration. 


Our prior foundational work \cite{Yuan_I_FamilyFloer} establishes that Maslov-0 holomorphic disks bounded by $\pi_0$-fibers provide a unique canonical algorithm for analytically gluing local segments $\trop^{-1}(V_i)\to V_i$ into a \textit{canonical dual affinoid torus fibration}, denoted by $\pi_0^\vee$. Notice that each fiber of $\trop$ is a copy of $U_\Lambda^n$. In geometric terms, the mirror analytic space $\mathscr Y_0$ admits a convenient set-theoretic description \cite{Yuan_I_FamilyFloer}:
\[
\mathscr Y_0 = \coprod_{q\in B_0} H^1(L_q; U_\Lambda) \cong \coprod_{q\in B_0} \{q\} \times U_\Lambda^n 
\]
where $L_q$ denotes the Lagrangian fiber over $q$ and $U_\Lambda$ is the unit circle in $\Lambda$ with respect to the non-archimedean norm.
This refers to item (iii) in Conjecture \ref{conjecture_SYZ}.

Despite the existence of a canonical dual fibration over the smooth locus $B_0$, current technology does not permit the canonical construction of dual singular fibers. Consequently, our approach relies on ad hoc constructions. For this, we hope to use fewer charts in the atlas $(U_i \to V_i)$ to achieve a clearer understanding of the integral affine structure on the smooth locus $B_0$. In the example studied in this paper, this is feasible with additional information about the locations of Maslov-0 disks (Proposition \ref{void_wall_cross_prop}).
Roughly, an atlas of integral affine structure will be produced by analyzing the local systems $\bigcup_{q\in B_0} \pi_2(X,L_q)$ and $\bigcup_{q\in B_0} \pi_1(L_q)$ on $B_0$.
Moreover, superpotentials can often reveal extra information for the analytic gluing by computing them in different chambers.
%


Finally, we explain the non-archimedean terms in Conjecture \ref{conjecture_SYZ}.
The \textit{tropically continuous map} is introduced by Chambert-Loir and Ducros \cite[3.1.6]{Formes_Chambert_2012} while our Definition \ref{tropically_continuous_defn} is slightly modified (see also the work of Gubler, Jell, and Rabinoff \cite{gubler2021forms}.)
In brief, this notion is reasonable because the reduced K\"ahler forms inevitably exhibit non-smoothness, which turns out to be related to the dual singular fibers.
A smooth point $p$ in the base of such a tropically continuous map $F$ can be defined as in \cite{KSAffine}, which means that there is an analytic isomorphism from $F^{-1}(V)\to V$ to $\trop^{-1}(V)\to V$ for a small neighborhood $V$ of $p$. If all points in the base are smooth, the map becomes the aforementioned affinoid torus fibration, inducing a natural integral affine structure on its base as well \cite[\S 4]{KSAffine}.
Remark that the work of Joyce \cite{Joyce_Singularity,Joyce_book} suggests that there are troubles to match the singular loci in the context of SYZ conjecture.
However, we will provide an explicit solution of fibration pair with matching singular loci.

\subsection{Explicit mirror construction}
The duality condition (iii) in Conjecture \ref{conjecture_SYZ} may require certain specialized knowledge of Lagrangian Floer theory (in families).
If we temporarily omit (iii), then the proof of Theorem \ref{Main_theorem_SYZ_intro} only require general knowledge since the mirror data can be expressed explicitly.
In \S \ref{s_mirror_construction}, we will provide a brief review of the family Floer mirror construction; we also suggest reading the outline in \S \ref{ss_outline_intro}.
For legibility, we focus on the algorithmic aspect, keeping the reliance on the Floer-theoretic foundation \cite{Yuan_I_FamilyFloer} to a minimum.

\begin{figure}
	\centering
	\includegraphics[width=5.2cm]{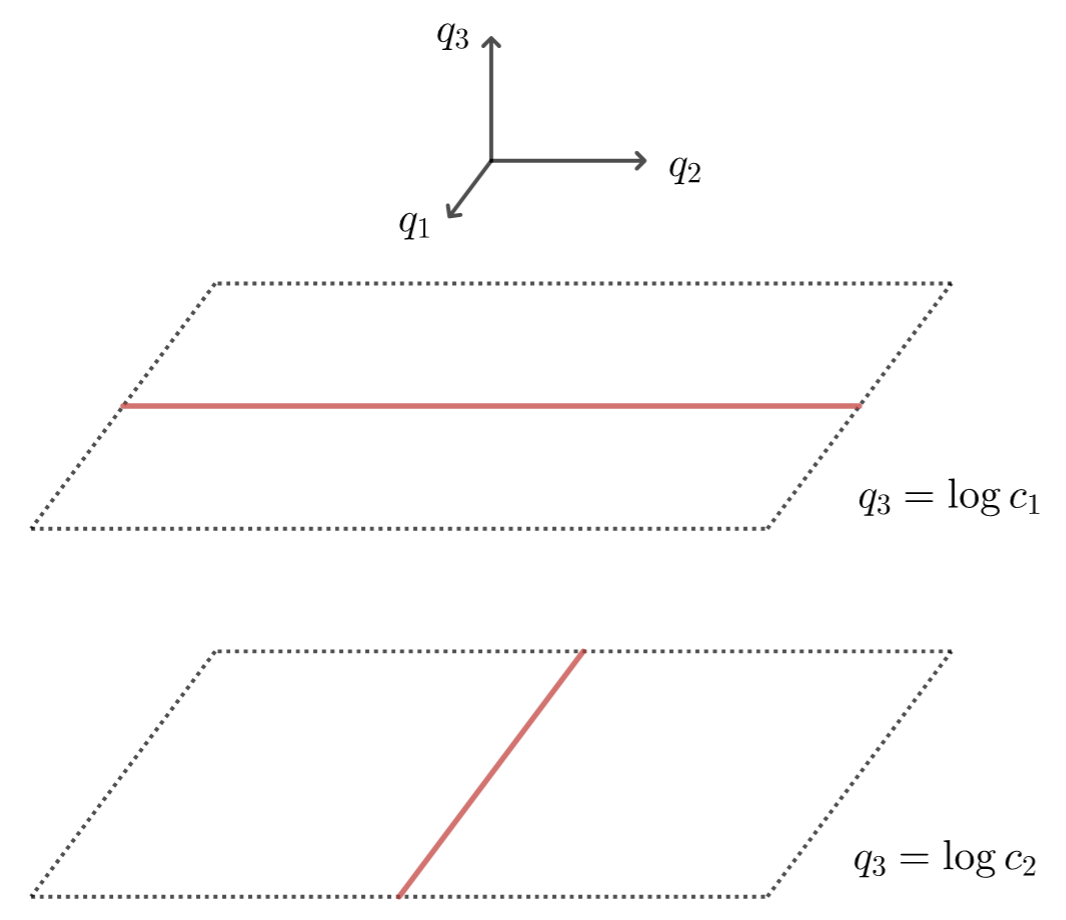}
	\caption{\small The singular locus consists of two skew lines in $\mathbb R^3$}
	\label{figure_singular_locus}
\end{figure}

\begin{proof}[Proof of Theorem \ref{Main_theorem_SYZ_intro} omitting (iii)]
	There is a special Lagrangian fibration in $X$ (cf. \cite{Chan_Pomerleano_Ueda}):
	\begin{equation}
			\label{pi_intro_eq}
		\pi =(\pi_1,\pi_2, \pi_3): X \to\mathbb R^3 , \qquad (u_1,u_2,v_1,v_2, z)\mapsto \Big( \frac{1}{2} (|u_1|^2-|v_1|^2), \frac{1}{2} (|u_2|^2-|v_2|^2), \log |z| \Big)
	\end{equation}
	Let $B_0$ and $\Delta$ denote the smooth and singular loci of $\pi$.
	It is known that $\Delta$ is given by the disjoint union of two skew lines (see Figure \ref{figure_singular_locus}):
	\begin{equation}
		\label{Delta_intro_eq}
	\Delta= \{0\}\times \mathbb R \times \{\log c_1\} \ \ \cup \ \ \mathbb R \times \{0\} \times \{\log c_2\}  \quad =:\Delta_1\cup \Delta_2
	\end{equation}
	There is a continuous map\footnote{For clarity, we always fix the choice of $\psi$ as in \S \ref{ss_action_coordinates}. It is actually not unique, and one can intentionally make various different choices of $\psi$. However, these differences are inessential.
	}
	$\psi : \mathbb R^3\to \mathbb R$, which is smooth in $B_0$, such that $(\pi_1,\pi_2, \psi)$ forms a set of the action coordinates locally over $B_0$.
	Define a non-archimedean analytic space $\mathscr Y=\{|x_2|<1\}$ in the analytification $Y^{\mathrm{an}}$ of $Y$.
	Define a continuous embedding $j:\mathbb R^3\to\mathbb R^5$ sending $q=(q_1,q_2,q_3)$ to
	\[
	(\theta_1(q),\theta_2(q),\vartheta(q), q_1,q_2)
	\]
where
\vspace{-1em}
\begin{align*}
		\theta_1(q) &=\min\{-\psi(q), -\psi(q_1,q_2,\log c_1) \} +\min\{0,q_1\}+\min\{0,q_2\}  \\
		\theta_2(q) &=\min\{ \ \ \ \psi(q),  \ \ \ \psi(q_1,q_2,\log c_2) \}   \\
		\vartheta(q) &=\median\{\psi(q), \ \ \ \psi(q_1,q_2,\log c_1), \ \ \ \psi( q_1,q_2,\log c_2) \}
\end{align*}
Define a tropically continuous map $F: Y^{\mathrm{an}}\to\mathbb R^5$ by
\[
F(x_1,x_2,z,y_1,y_2)=(F_1,F_2,G,\val(y_1),\val(y_2))
\]
where
\vspace{-1em}
	\begin{align*}
		F_1&=\min\{\val(x_1), -\psi(\val(y_1),\val(y_2), \log c_1) +\min\{0,\val(y_1)\}+\min\{0,\val(y_2)\} \} \\
		F_2&=\min\{\val(x_2),  \ \ \ \psi(\val(y_1),\val(y_2),\log c_2) \} \\
		G&=\median\{\val(z)+\min\{0,\val(y_2)\} , \ \ \ \psi(\val(y_1),\val(y_2),\log c_1), \ \ \ \psi(\val(y_1),\val(y_2), \log c_2) \} 
	\end{align*}
Then, we can verify that $j(\mathbb R^3)=F(\mathscr Y)$. So, we can define (cf. \cite[\S 8]{KSAffine})
\begin{equation}
	\label{f_intro_one_page_proof}
f=j^{-1} \circ F : \mathscr Y \to\mathbb R^3	
\end{equation}
Finally, we will find: the smooth / singular loci and the induced integral affine structure of $f$ all precisely agree with those of $\pi$ in (\ref{pi_intro_eq}). (The full details will be presented in the main body of this paper.)
\end{proof}

\begin{rmk}
	\label{j_minor_role}
	While the mirror \textit{space} is known to be the conifold resolution, realizing the mirror \textit{fibration} within it is a separate challenge and appears to be unknown prior to this work. The conifold resolution equations $x_1z=1+y_1$ and $x_2=(1+y_2)z$, given in (\ref{Y_mirror_alg_var_eq_intro}), impose rigid restrictions on non-archimedean valuations $\val(x_1),\val(x_2),\val(y_1),\val(y_2),\val(z)$ of all variables. Properly arranging these valuations is essential for accommodating the integral affine structure and singular locus. The expressions of solutions may not be unique, as one can replace $(j, F)$ with another pair $(\phi(j), \phi(F))$ for any automorphism $\phi$. The role of $j$ is also relatively minor, as it is also fine to compare $j\circ \pi$ and $F$ instead of $\pi$ and $f=j^{-1}\circ F$.
\end{rmk}

\begin{rmk}
The median might represents a non-archimedean partition of unity.
For the Berkovich topology on the affine line with variable $x$, there is a bump function $\phi(x)=\mathrm{median}\{\val(x),0,1\}$. In $\{\val(x) < 0\}=\{|x|>1\}$ and $\{\val(x)> 1\}=\{|x|<1/e\}$, this $\phi(x)$ takes values 0 and 1. While this idea is moderately more systematically expanded in \cite{Yuan_An}, the approach of this paper tends to be ad hoc, largely due to the fact that addressing 3-fold singularities is generally harder than that of surfaces (cf. \cite[3.6]{seidel2001braid}).
Many homological mirror symmetry results are highly valued, despite their reliance on ad hoc strategies.
One of the main contributions of this paper is the introduction of a quite explicit example, which enriches the sparse set of examples substantiating the fibration duality aspect of the SYZ conjecture.
The task of "finding a solution with desired properties" is very different from "checking if a given solution satisfies desired properties".
The latter is valueless, but we address the former task.
\end{rmk}

\begin{rmk}
It is not always possible to fully justify why a specific ad hoc construction, like the medians in (\ref{f_intro_one_page_proof}), is chosen and preferred over another.
In the early stages of the theory, we lack developed tools, which leaves us with no choice but to search for solutions {by hand}. We apologize if this process may be tedious and unmotivated. However, examples with fully explicit formulas prevent our mirror construction method \cite{Yuan_I_FamilyFloer} from becoming a "castle in the air."
Besides, there are certain clues from Floer-theoretic basis as sketched in \S \ref{s_mirror_construction}; no matter how peculiar our solution appears in the first glance, it {explicitly} represents a new instance of the SYZ conjecture, including singular fibers. The credibility is further bolstered by the extra evidence in (\ref{cluster_SYZ_phenomenon}).
\end{rmk}

The structure of the proof can be briefly outlined as follows.

\begin{itemize}[topsep=1pt, itemsep=1pt, parsep=0pt]
	\item In \S \ref{s_A_side}, we review the Lagrangian fibration $\pi$ on the conifold smoothing $X$. We specify the atlas of integral affine structure in explicit terms of the symplectic areas of holomorphic disks. Studying the monodromy of these disks is required for a global grasp of the integral affine structure on $B_0$.
	
	\item In \S \ref{s_mirror_construction}, we review the family Floer mirror construction with a focus on the algorithmic aspect. Then, we show an identification between the abstract non-archimedean mirror space structure on $R^1\pi_{0*}(U_\Lambda)$ and a concrete analytic space (\ref{T_123_glue}), obtained by gluing three analytic open domains $T_1, T_2, T_3$ in the torus $(\Lambda^*)^3$ equipped with Berkovich analytic topology.
	
	\item In \S \ref{s_B_side}, we construct a tropically continuous fibration $F$ on the analytification of conifold resolution $Y$ and examine its image space in $\mathbb R^5$ (Theorem \ref{F_affinoid_thm}).
	We also determine the smooth and singular loci of $F$.
	Remark that $F$ is somewhat ad hoc but is carefully designed to successfully ensure the T-duality matching.
	
	\item In \S \ref{s_T_duality_matching}, we define an analytic embedding map $g$ into the conifold resolution $Y$ and verify that the $g$ intertwines the canonical dual affinoid torus fibration $\pi_0^\vee: R^1\pi_{0*}(U_\Lambda)\to B_0$ and the smooth part of the above $F$ up to a homeomorphism $j$ between their base spaces. This matches the integral affine structure and enables the singular extension based on the above study of $F$.
\end{itemize}

\subsection{Additional evidence}
\label{sss_speculation_Chan}

We recognize that verifying the shared properties between the two fibrations $\pi$ and $f$ can be laborious and wearisome, notably in Theorem \ref{F_affinoid_thm}. To ensure conviction before engaging in computations, let's present some easily verifiable evidence for the reader. See (\ref{cluster_SYZ_phenomenon}).

It is interesting to note that the examples of Conjecture \ref{conjecture_SYZ} are often related to certain cluster varieties on A-side and/or B-side.
We recall a {cluster variety} generally has a collection of charts equivalent to the algebraic torus $(\mathbb C^*)^n$ glued by birational transition maps (see \cite{GHK_birational}). As mentioned before, we are freed from the constraint of the special Lagrangian condition and only need a more topological condition, graded or zero Maslov class.
We expect that such Lagrangians should appear abundant, which seems to agree with the result of M. Castronovo \cite{castronovo2019exotic}.
We might imagine there would be such Lagrangian fibrations in each algebraic torus chart of a cluster variety in a compatible way so that a global fibration could be produced. Thus, we expect that the next examples of Conjecture \ref{conjecture_SYZ} would be found for some cluster varieties such as Grassmannians. We defer this to future studies.

For $j=1,2$, the component $\pi_j$ in (\ref{pi_intro_eq}) is the moment map of the $S^1$-action $(u_j,v_j)\mapsto (e^{it}u_j, e^{-it}v_j)$, and the fixed point sets of them are given by
\begin{equation}
	\label{C_1_C_2_hat_eq}
	\begin{aligned}
		\hat C_1 &:=\{  u_1=v_1=0, \ z=-c_1  \} \\
		\hat C_2 &:=\{ u_2=v_2=0, \ z=-c_2  \}
	\end{aligned}
\end{equation}
Then, by (\ref{Delta_intro_eq}), we can check that
\begin{equation}
	\label{cluster_SYZ_A_side_phenomenon}
	\pi(\hat C_1)=\Delta_1 \qquad \text{and} \qquad \pi (\hat C_2)=\Delta_2
\end{equation}

The use of the Novikov field $\Lambda=\mathbb C((T^{\mathbb R}))$ is essential for the sake of quantum correction.
However, we want to further demonstrate that working over $\Lambda$ has great advantages even beyond the symplectic geometry scopes.
To see this, let's consider the same variety as (\ref{Y_mirror_alg_var_eq_intro}) over $\mathbb C$. Namely, we take (cf. \cite[\S 2]{Chan_Pomerleano_Ueda})
\[
	Y(\mathbb C)= \left\{ (x_1, x_2, z, y_1,y_2)\in  \mathbb C^2\times \mathbb {CP}^1 \times (\mathbb C^*)^2  \ \  \mid \ \  
	\begin{aligned}
		x_1 z &=1+y_1 \\
		x_2 &=(1+y_2) z
	\end{aligned}
	\right\}
\]

\begin{rmk}
	Note that $Y(\mathbb C)$ is contained in $Z^0=\{ (x_1,x_2,y_1,y_2) \in\mathbb C^2\times (\mathbb C^*)^2 \mid x_1x_2=(1+y_1)(1+y_2)\}$. Meanwhile, Chan, Pomerleano, and Ueda in \cite[Corollary A.5]{Chan_Pomerleano_Ueda} prove that the symplectic cohomology $SH^0(X)$ is isomorphic to the coordinate ring of $Z^0$. (Pascaleff \cite{pascaleff2019symplectic} also obtains some similar results for log Calabi-Yau surfaces.)
	This computation should be the evidence for a closed-string approach of mirror construction developed by Groman and Varolgunes \cite{groman_varolgunes_2022closed_string}, and its relation to our open-string approach should be achieved by some version of closed-open maps, which will be discussed elsewhere.
\end{rmk}

Consider the three Zariski open charts in $Y(\mathbb C)$ given by
\begin{equation}
	\label{mathcal_T_k_eq}
\mathcal T_1=\{x_1\neq 0\} \ ,  \ \  \mathcal T_2=  \{0\neq z\neq \infty\}   \  ,  \ \ \text{and} \ \mathcal T_3= \{x_2\neq 0\}
\end{equation}
all of which are algebraically equivalent to $(\mathbb C^*)^3$ and cover the complement $Y(\mathbb C)\setminus (C_1\cup C_2)$
where
\begin{equation}
	\label{C_1_C_2_eq}
	\begin{aligned}
	C_1&=\{ x_1=x_2=0, y_1=-1, z=0 \} \\
	C_2&=\{ x_1=x_2=0, y_2=-1, z=\infty \}
\end{aligned}
\end{equation}
are of codimension 2 in $Y(\mathbb C)$.
In general, a cluster variety over $\mathbb C$ is usually only covered by the algebraic torus charts up to codimension two (see e.g. the work of Gross, Hacking, and Keel in \cite[\S 3.2]{GHK_birational}).

The speculation of Chan, Pomerleano, and Ueda in \cite[Remark 2.5]{Chan_Pomerleano_Ueda} is that the codimension-2 `missing points' in $C_1\cup C_2$ can be possibly understood by the dual SYZ singular fibers in some way.
Now, we verify it and show, however, that it has to be understood in the non-archimedean world!

Let's go back to the  variety $Y$ in (\ref{Y_mirror_alg_var_eq_intro}) over the non-archimedean Novikov field $\Lambda=\mathbb C((T^{\mathbb R}))$.
By considering the analytification $Y^{\mathrm{an}}$ of $Y$, we can take advantage of the non-archimedean analytic topology that is finer than the Zariski topology.
For instance, instead of the algebraic torus in (\ref{mathcal_T_k_eq}), our mirror analytic space is obtained by gluing three \textit{analytic} open subsets $T_1,T_2,T_3$ (see \S \ref{ss_g}) that are strictly contained in the \textit{Zariski} open subsets $\mathcal T_1,\mathcal T_2,\mathcal T_3$ in (\ref{mathcal_T_k_eq}) respectively.

We abuse the notations and still denote by $C_1$ and $C_2$ the sub-varieties in $Y$ over $\Lambda$ of the same equations in (\ref{C_1_C_2_eq}).
Using the singular locus $\Delta=\Delta_1\cup\Delta_2$ in (\ref{Delta_intro_eq}) and the dual singular analytic fibration map $f$ in (\ref{f_intro_one_page_proof}), it is straightforward to check that
$j(\Delta_1) =F(C_1)$ and $j(\Delta_2)=F(C_2)$, hence,
\begin{equation}
	\label{cluster_SYZ_phenomenon}
f(C_1)=\Delta_1 \qquad \text{and} \qquad f(C_2)=\Delta_2
\end{equation}
\vspace{-2em}
\begin{proof}[\underline{Sketch of computation}]
	Let $(x_1,x_2,y_1,y_2,z)$ be an arbitrary point in $C_1$. Then, $\val(x_1)=\val(x_2)=\val(z)=+\infty$ and $\val(y_1)=0$.
	Set $s=\val(y_2)\in\mathbb R$.
For the $F_1, F_2, G$ in the formula (\ref{f_intro_one_page_proof}), we obtain that
		\begin{align*}
		F_1&=\min\{+\infty , -\psi(0, s , \log c_1) +\min\{0,s\} \} &&= -\psi(0, s , \log c_1) +\min\{0,s\}  \\
		F_2&=\min\{+\infty,  \ \ \ \psi(0,s,\log c_2) \} &&= \ \ \ \psi(0,s,\log c_2) \\
		G&=\median\{+\infty , \ \ \ \psi(0,s, \log c_1), \ \ \ \psi(0,s, \log c_2) \} &&= \ \ \ \psi(0,s, \log c_1)
	\end{align*}
Notice $\psi$ increases with the last variable and $c_1>c_2$.
This describes $F(C_1)$, and we next check $j(\Delta_1)$. 
Note that $\Delta_1$ is parameterized by $(0,s, \log c_1)$ for $s\in\mathbb R$.
For the $\theta_1,\theta_2,\vartheta$ in (\ref{f_intro_one_page_proof}), we gain that
\begin{align*}
	\theta_1 &= -\psi(0,s,\log c_1)  +\min\{0,s\}  \\
	\theta_2 &= \ \ \ \psi(0,s,\log c_2)    \\
	\vartheta &=\median\{\psi(0,s,\log c_1), \ \ \ \psi(0,s,\log c_1), \ \ \ \psi( 0,s,\log c_2) \} = \psi(0,s,\log c_1) 
\end{align*}
Therefore, we have verified that $f(C_1)=\Delta_1$. Similarly, one can easily verify that $f(C_2)=\Delta_2$.
\end{proof}


Although the construction of $f$ is inspired by our version of SYZ T-duality in Conjecture \ref{conjecture_SYZ}, 
the above phenomenon is also interesting
merely on the B-side, forgetting the A-side and mirror symmetry background. In particular, we do not have to work over $\Lambda=\mathbb C((T^{\mathbb R}))$.
In view of (\ref{cluster_SYZ_phenomenon}), we can propose a concrete question purely within the area of non-archimedean Berkovich geometry:

\begin{question}
	Let $V$ be a cluster variety over an arbitrary non-archimedean field $\Bbbk$. Assume $C$ is a codimension-2 sub-variety in $V$ such that $V\setminus C$ can be covered by a collection of algebraic torus charts. Does there exist a Zariski-dense analytic open subset $\mathscr V$ in $V^{\mathrm{an}}$
	and a tropically continuous map $f:\mathscr V\to B$ onto a topological manifold $B$ such that $\Delta:=f(C)$ is exactly the singular locus of $f$?
\end{question}

We conjecture that the answer is positive for any cluster variety of finite type (in the sense that $V\setminus C$ is covered by a finite collection of algebraic torus charts).
At least, it holds for the example of this paper (\ref{Y_mirror_alg_var_eq_intro}) as well as all the examples in \cite{Yuan_local_SYZ} by similar explicit computations.
Hopefully, by studying the above question, we could discover many interesting relations among SYZ conjecture, symplectic geometry, non-archimedean geometry, and cluster structures.

Finally, we indicate that $f$ solves Conjecture \ref{conjecture_SYZ} for the Lagrangian fibration $\pi$ in (\ref{pi_intro_eq}).
It is very intriguing to put (\ref{cluster_SYZ_A_side_phenomenon}) and (\ref{cluster_SYZ_phenomenon}) together, which is a striking coincidence.
The matching of singular loci of A-side and B-side seems to be related to the cluster theory in a very concrete way.
Hopefully, we could formulate this phenomenon for the singular fibers on both sides more precisely in the near future.


\vspace{1em}
\textbf{Acknowledgment . }
The author is deeply grateful to the referee for their meticulous efforts in reviewing the paper, particularly for identifying various typos that would have gone unnoticed without their careful and thorough examination.
The author is also grateful to Kenji Fukaya and Eric Zaslow for their constant support and to Mohammed Abouzaid, Siu-Cheong Lau, Wenyuan Li, and Vivek Shende for helpful conversations.

\section{Lagrangian fibration on the deformed conifold: A side}
\label{s_A_side}

\subsection{Lagrangian fibration}
\label{ss_Gross_fib}
Let $c_1>c_2>0$ be fixed positive real numbers.
Define
\[
X=\{ (u_1,u_2,v_1,v_2, z)\in \mathbb C^4\times \mathbb C^* \mid u_1v_1=z+c_1, \quad u_2v_2=z+c_2 \}
\]
which can be regarded as an open subset contained in the smoothing variety
\[
X'= \{ (u_1,u_2,v_1,v_2)\in\mathbb C^4 \mid u_1v_1-c_1=u_2v_2-c_2\}
\]
of the conifold $\{u_1v_1=u_2v_2\}$.
For the divisor
$
\mathscr D=\{ u_1v_1=c_1\}=\{u_2v_2=c_2\}
$
in $X'$, we have 
\[
X=X'\setminus \mathscr D
\]
Alternatively, we interpret $X$ (resp. $X'$) as the fiber product of the two maps
$
f_i:=u_iv_i-c_i: \mathbb C^2\to\mathbb C_z^*$ (resp. $\mathbb C^2 \to \mathbb C_z$) for $i=1,2$ fitting in the following diagram:
\begin{equation}
	\label{fiber_product_X+_eq}
\xymatrix{
	&\,  \ X' \ar[dl]_{g_1} \ar[dr]^{g_2}  \\
	\mathbb C^2_{u_1,v_1} \ar[dr]_{f_1}& &  \mathbb C^2_{u_2,v_2}  \ar[dl]^{f_2}& \\
	 & \mathbb C_z
}
\end{equation}
We equip $X$ with the restriction $\omega=d\lambda$ of the standard symplectic form $\omega_0$ on $\mathbb C^4\times \mathbb C^*$, namely,
\[
\textstyle
\omega_0= \sum_{k=1,2}  ( \frac{i}{2} du_k\wedge d\bar u_k +  dv_k\wedge d\bar v_k) + \frac{i}{2} \frac{dz\wedge d\bar z}{|z|^2}
\]
The following is a special Lagrangian fibration
\[
\textstyle
\pi =(\pi_1,\pi_2, \pi_3): X \to\mathbb R^3 , \qquad (u_1,u_2,v_1,v_2, z)\mapsto \Big( \frac{1}{2} (|u_1|^2-|v_1|^2), \frac{1}{2} (|u_2|^2-|v_2|^2), \log |z| \Big)
\]
with respect to the holomorphic volume form $\Omega=d\log z\wedge d\log u_1\wedge d\log u_2$ (cf. \cite{Chan_Pomerleano_Ueda}).
Denote by $L_q$ the Lagrangian fiber over $q=(q_1,q_2,q_3)$ in $\mathbb R^3$. By (\ref{fiber_product_X+_eq}), the $L_q$ can be also viewed as the fiber product
of $L_{1,\hat q_1}$ and $L_{2,\hat q_2}$
where $\hat q_1=(q_1,q_3)$, $\hat q_2=(q_2,q_3)$, and
\begin{equation}
	\label{fiber_product_L_q_eq}
L_{i,\hat q_i}=\{(u_i,v_i)\in\mathbb C^2\mid \tfrac{1}{2}(|u_i|^2-|v_i|^2)=q_i, \quad   |u_iv_i-c_i|=\exp (q_3)\}
\end{equation}
Clearly, $(\pi_1,\pi_2)$ is the moment map of the Hamiltonian $T^2$-action given by
\begin{equation}
	\label{Ham_act_eq}
(e^{is}, e^{it})\cdot (u_1,u_2,v_1,v_2,z)\mapsto (e^{is} u_1, e^{it} u_2, e^{-is}v_1, e^{-it} v_2, z)
\end{equation}
The action degenerates when $u_1=v_1=0$ or $u_2=v_2=0$.
The \textit{discriminant locus} of $\pi$ is given by
$
\Delta=\Delta_1\cup \Delta_2
$
where
\begin{equation}
	\label{Delta_singular_locus_eq}
\Delta_1=\{0\}\times \mathbb R_{q_2} \times \{ \log c_1\} 
\quad \text{and} \quad
\Delta_2=\mathbb R_{q_1}  \times \{0\} \times \{\log c_2\}
\end{equation}
are two skew lines in $\mathbb R^3$.
Then, the smooth locus of $\pi$ is given by 
\[
B_0:=\mathbb R^3\setminus \Delta
\]
We also write $X_0=\pi^{-1}(B_0)$ and 
\begin{equation}
	\label{pi_0}
\pi_0:=\pi|_{B_0}:X_0\to B_0
\end{equation}
It is well-known that a Lagrangian torus fiber $L_q$ for $q=(q_1,q_2,q_3)\in B_0$ bounds a nontrivial Maslov-0 holomorphic disk in $X$ if and only if
$q_3=\log c_1$ or $q_3=\log c_2$ (see e.g. \cite[Proposition 2.2]{Chan_Pomerleano_Ueda}).
In other words, we have 4 components of the walls of Maslov-0 holomorphic disks as follows: (cf. Figure \ref{figure_singular_locus})
\begin{equation}
	\label{walls_H_ipm_eq}
H_{1\pm}
= \mathbb R_\pm \times \mathbb R \times \{\log c_1\}
\quad \text{and}\quad
H_{2\pm}
= \mathbb R \times \mathbb R_\pm \times \{\log c_2\}
\end{equation}
where we set $\mathbb R_+=(0,+\infty)$ and $\mathbb R_-=(-\infty, 0)$.
For later uses, we also introduce the notations:
\[
\bar H_{1\pm}
= \mathbb {\bar R}_\pm \times \mathbb R \times \{\log c_1\}
\quad \text{and}\quad
\bar H_{2\pm}
= \mathbb R \times \mathbb {\bar R_\pm} \times \{\log c_2\}
\]
where we set $\mathbb {\bar R}_+=[0,+\infty)$ and $\mathbb {\bar R}_-=(-\infty, 0]$.

If we define the divisors
$
D_{i-}=\{ u_i=0\} \quad \text{and} \quad D_{i+}=\{v_i=0\}
$
in $X$ for $i=1,2$, then we observe that for $i=1,2$, 
$ \pi(D_{i-})= \bar H_{i-} $, $\pi(D_{i+})=\bar H_{i+}$,
and
$\pi (D_{i+}\cap D_{i-})= \Delta_i$.

\subsection{Topological disks}
\label{ss_topological_disks}
We consider the following local systems over $B_0$:
\begin{equation}
	\label{local_system_H_2_H_1_eq}
	\mathscr R_1:=R^1\pi_*(\mathbb Z)\equiv \bigcup_{q\in B_0} \pi_1(L_q), \qquad  \mathscr R_2 :=\mathscr R_2(X'):=\bigcup_{q\in B_0} \pi_2( X',L_q)
\end{equation}
Abusing the notations, the fibers $\pi_1(L_q)$ and $\pi_2(X',L_q)$ of $\mathscr R_1$ and $\mathscr R_2$ actually denote the corresponding images of the Hurewicz maps in the (relative) homology groups $H_1(L_q)$ and $H_2(X',L_q)$ respectively rather than the homotopy groups. We apologize for the unusual notations, and we just attempt to avoid using $H_1(L_q)$ and $H^1(L_q)$ in the same time.

\begin{rmk}
	For our purpose, we must fully understand the monodromy of the local systems $\mathscr R_2$ and $\mathscr R_1$. The relevant literature, like \cite{Chan_Pomerleano_Ueda}, typically only addresses a single wall without delving into the monodromy. Therefore, we need to provide the complete details in the following text.
\end{rmk}

Let $\mathscr N_{i\pm}$ be small neighborhoods of $H_{i\pm}$ in $B_0$ for $i=1,2$.
The complement $B_0\setminus \bigcup_{i=1,2} H_{i+}\cup H_{i-}$ has three components: (cf. Figure \ref{figure_singular_locus})
\begin{align*}
B_1 
&= \{q\in B_0\mid \log c_1<q_3<+\infty \} \\
B_2
&= \{q\in B_0\mid  \log c_2<q_3<\log c_1\} \\
B_3
&=\{q\in B_0\mid -\infty <q_3<\log c_2\}
\end{align*}
in which we recall that $c_1>c_2>0$.
Then, we consider the slight thickenings of them as follows:
\begin{equation}
	\label{U_123_eq}
	\begin{aligned}
	U_1 
	&= B_1\cup \mathscr N_{1+} \cup \mathscr N_{1-} \\
	U_2
	&= B_2\cup \mathscr N_{2+}\cup\mathscr N_{2-} \\
	U_3
	&=B_3
	\end{aligned}
\end{equation}
The above three open subsets form a covering of $B_0$. Moreover, they are contractible over which the local systems $\mathscr R_1$ and $\mathscr R_2$ can be trivialized.

Let $q=(q_1,q_2,q_3)\in B_0$.
Fix $\pmb \beta\in\pi_2(X', L_q)$, and let $u:(\mathbb D,\partial\mathbb D)\to (X', L_q)$ be a \textit{topological} disk that represents $\pmb \beta$.
The map $u$ can be identified with the fiber product of two maps $u_i:(\mathbb D,\partial\mathbb D)\to (\mathbb C^2, L_{i, \hat q_i})$ for $i=1,2$. Set $\beta_i=[u_i]$, and we denote
$
\pmb \beta=(\beta_1,\beta_2)
$.

Now, we study the following three cases (cf. \cite[\S 2]{Chan_Pomerleano_Ueda}):

\begin{itemize}
	\item If $q\in B_3$, then for each $i=1,2$, the $L_{i,\hat q_i}$ is the Chekanov torus in $\mathbb C^2$. There is a preferred topological class $\hat\beta_i$ in $\pi_2(\mathbb C^2, L_{i,\hat q_i})$ of Maslov index 2. Denote by
	\[
	\pmb \beta_3 :=(\hat\beta_1,\hat\beta_2)
	\]
 the induced class in $\pi_2(X',L_q)$. Its open Gromov-Witten invariant is known to be $\mathsf n_{\pmb\beta_3}=1$.
	
	\item If $q\in B_2$, then $L_{1,\hat q_1}$ is of Chekanov type while $L_{2,\hat q_2}$ is of Clifford type.
	There are two Maslov-2 classes $\beta_{2\pm}$ in $\pi_2(\mathbb C^2, L_{2,\hat q_2})$ such that $\beta_{2+}\cdot \{v_2=0\}=1$ and $\beta_{2-}\cdot \{u_2=0\}=1$ in $\mathbb C^2_{u_2,v_2}$. Write
	\[
	\pmb \beta_{2+}:= (\hat\beta_1, \beta_{2+}) \quad \text{and} \quad \pmb \beta_{2-}:= (\hat\beta_1, \beta_{2-})
	\]
	The corresponding open Gromov-Witten invariants are also known: $\mathsf n_{\pmb \beta_{2\pm}}=1$.

	\item If $q\in B_1$, then for each $i=1,2$, the $L_{i, \hat q_i}$ is of Clifford type, and there are two Maslov-2 classes $\beta_{i\pm}$ in $\pi_2(\mathbb C^2, L_{i,\hat q_i})$ such that $\beta_{i+}\cdot \{v_i=0\}=1$ and $\beta_{i-}\cdot \{u_i=0\}=1$ in $\mathbb C^2_{u_i,v_i}$. We write
	\[
	\pmb\beta_{1++}:=(\beta_{1+},\beta_{2+}), \quad \pmb\beta_{1-+}:=(\beta_{1-},\beta_{2+}), \quad
	\pmb\beta_{1+-}:=(\beta_{1+},\beta_{2-}), \quad \pmb\beta_{1--}:=(\beta_{1-},\beta_{2-})
	\]
	Similarly, the corresponding open Gromov-Witten invariants are $\mathsf n_{\pmb \beta_{1\pm\pm}}=1$.
	It is standard that
	\[
	\pmb\beta_{1-+}-\pmb\beta_{1++} =\pmb\beta_{1--}-\pmb\beta_{1+-}
	\]
\end{itemize}

Next, we study the monodromy among all these topological disks. To see this, it suffices to study their relations over the small neighborhoods $\mathscr N_{i\pm}$.
	For (\ref{local_system_H_2_H_1_eq}) and (\ref{U_123_eq}), we can further regard the $\pmb \beta_3$ as a section of $\mathscr R_2$ over the contractible $U_3$, the $\pmb\beta_{2\pm}$ as sections over $U_2$, and the $\pmb\beta_{1\pm\pm}$ as sections over $U_1$.
	Beware that it is just in the topological level.
As in \cite{Chan_Pomerleano_Ueda}, one can check the following relations of these topological sections over the neighborhoods $\mathscr N_{i\pm}$ of the walls:
\begin{equation}
	\label{beta_32_eq}
	\begin{cases}
	\pmb \beta_3=\pmb \beta_{2+}   & \text{over} \ \mathscr N_{2+} \\
	\pmb \beta_3=\pmb \beta_{2-}	& \text{over} \ \mathscr N_{2-} \\
	\end{cases}
\end{equation}
\begin{equation}
	\label{beta_21_eq}
	\begin{cases}
		\pmb \beta_{2+} = \pmb \beta_{1++} & \text{over} \ \mathscr N_{1+}  \\
		\pmb \beta_{2+} = \pmb \beta_{1-+} & \text{over} \ \mathscr N_{1-}  \\
		\pmb \beta_{2-} = \pmb \beta_{1+-} & \text{over} \ \mathscr N_{1+}  \\
		\pmb \beta_{2-} = \pmb \beta_{1--} & \text{over} \ \mathscr N_{1-}  \\
	\end{cases}
\end{equation}

From now on, our convention is to write $E(\pmb \beta)=\frac{1}{2\pi} \int_{\pmb \beta}  \omega$ for a disk class $\pmb \beta$.

For (\ref{local_system_H_2_H_1_eq}), there is a natural boundary map $\partial:\mathscr R_2\to\mathscr R_1$ which is fiberwise $\pi_2(X',L_q)\to\pi_1(L_q)$.
Given $q\in B_0$ and $i=1,2$, we use $\sigma_i=\sigma_i(q)$ to denote the class of the orbit of the $i$-th $S^1$-component of the action in (\ref{Ham_act_eq}). They can be regarded as the global sections of $\mathscr R_1$.
Moreover, we can check the following agreements of topological sections of $\mathscr R_1$ (cf. \cite{Yuan_local_SYZ,Chan_Pomerleano_Ueda})
\begin{equation}
	\label{sigma_12_relations_eq}
\begin{cases}
	\sigma_2= \partial \pmb \beta_{2-}- \partial \pmb\beta_{2+}  & \text{over } \ U_2  \\
\sigma_2=	\partial \pmb \beta_{1+-}-\partial\pmb\beta_{1++} =\partial\pmb \beta_{1--}-\partial\pmb \beta_{1-+}   & \text{over } \ U_1 \\
\sigma_1 = \partial \pmb\beta_{1-+}  -  \partial\pmb\beta_{1++} =  \partial \pmb\beta_{1--} - \partial\pmb\beta_{1+-} 
& \text{over } \ U_1
\end{cases} 
\end{equation}
For example, we know that over $U_1$, 
\begin{equation}
	\label{sigma_1sigma_2_eq}
	\partial\pmb \beta_{1--}=\sigma_1+\partial\pmb \beta_{1+-}= \sigma_1+\sigma_2+\partial\pmb\beta_{1++}
\end{equation}

\subsection{Action coordinates}
\label{ss_action_coordinates}
Following \cite{Yuan_local_SYZ}, we specify the action coordinates by choosing the bases of $\pi_1(L_q)$'s and taking the flux maps.
Of course, one can intentionally make multiple different choices, but for clarity, we work with the fixed choices as follows:

\begin{itemize}
\item We choose a frame of $\mathscr R_1|_{U_3}$ to be the sections $\{\sigma_1,\sigma_2, \partial \pmb \beta_3 \}$.
Then, an induced integral affine coordinate chart is given by
\[
	\chi_3:  U_3 \to\mathbb R^3, \qquad  q=(q_1,q_2,q_3)\mapsto (q_1, q_2, \psi_3(q))
\]
where we write
\[
\textstyle
\psi_3(q)= \frac{1}{2\pi} \int_{\pmb \beta_3} \omega = E(\pmb \beta_3)
\]

\item We choose a frame of $\mathscr R_1|_{U_2}$ to be the sections $\{\sigma_1,\sigma_2, \partial \pmb\beta_{2+}\}$. Then, an induced integral affine coordinate chart is given by
\[
\chi_2:U_2\to\mathbb R^3, \qquad q=(q_1,q_2,q_3)\mapsto  (q_1, q_2, \psi_2(q))
\]
where we write
\[
\textstyle
\psi_2(q)= \frac{1}{2\pi} \int_{\pmb \beta_{2+}} \omega = E(\pmb \beta_{2+})
\]

\item We choose a frame of $\mathscr R_1|_{U_1}$ to be the sections $\{\sigma_1,\sigma_2, \partial \pmb\beta_{1++}\}$. Then, an induced integral affine coordinate chart is given by
\[
\chi_1:U_1\to\mathbb R^3, \qquad q=(q_1,q_2,q_3)\mapsto  (q_1, q_2, \psi_1(q))
\]
where we write
\[
\textstyle
\psi_1(q)= \frac{1}{2\pi} \int_{\pmb\beta_{1++}} \omega = E(\pmb \beta_{1++} )
\]
\end{itemize}

The first two coordinates $q_1,q_2$ of $\chi_k$ for $k=1,2,3$ can be made the same since they correspond to the moment map $(\pi_1,\pi_2)$ of the $T^2$-action in (\ref{Ham_act_eq}).
Recall that the symplectic form $\omega=d\lambda$ is the standard one.
Due to (\ref{sigma_12_relations_eq}), $q_2=\frac{1}{2\pi} \int_{\sigma_2}\lambda=\frac{1}{2\pi} \int_{\partial\pmb\beta_{2-}-\partial\pmb\beta_{2+}} \lambda $, and by Stokes' formula we get that
\[
\begin{cases}
	q_2 = E(\pmb\beta_{2-})-E(\pmb\beta_{2+})  & \text{over} \ U_2 \\
	q_2 = E( \pmb \beta_{1+-} ) -  E(\pmb\beta_{1++}) =E(\pmb \beta_{1--}) -  E(\pmb \beta_{1-+})   & \text{over } \ U_1
\end{cases}
\]
Similarly, we also conclude that
\[
q_1=E(\pmb\beta_{1-+})-E(\pmb\beta_{1++})=E(\pmb\beta_{1--})-E(\pmb\beta_{1+-}) \qquad \text{over} \ U_1
\]
For example, we know that 
\begin{equation}
	\label{q1+q2_eq}
E(\pmb\beta_{1--})=q_1+E(\pmb\beta_{1+-})=q_1+q_2+E(\pmb\beta_{1++})
\end{equation}
over $U_1$.
Applying (\ref{beta_32_eq}) and (\ref{beta_21_eq}) further deduces the integral affine transformations among the $\chi_k$'s.
Specifically, we can check that
\begin{equation}
	\label{psi_32_eq}
	\begin{cases}
		\psi_3(q)=\psi_2(q)   & \text{over} \ \mathscr N_{2+} \\
		\psi_3(q)=\psi_2(q)+q_2	& \text{over} \ \mathscr N_{2-} \\
	\end{cases}
\end{equation}
\begin{equation}
	\label{psi_21_eq}
	\begin{cases}
		\psi_2(q) = \psi_1(q) & \text{over} \ \mathscr N_{1+}  \\
		\psi_2(q) = \psi_1(q)+q_1  & \text{over} \ \mathscr N_{1-}  \\
	\end{cases}
\end{equation}
In other words,
\begin{align*}
	\psi_3(q)=\psi_2(q)+\min\{0,q_2\}  \qquad & \text{over} \ \mathscr N_{2+}\cup \mathscr N_{2-} \\
	\psi_2(q)=\psi_1(q)+\min\{0,q_1\} \qquad & \text{over} \ \mathscr N_{1+}\cup\mathscr N_{1-}
\end{align*}
Therefore, we are able to define a smooth function on $B_0=U_1\cup U_2\cup U_3=\mathbb R^3\setminus \Delta$ as follows
\begin{equation}
	\label{psi_123_together_eq}
	\psi=\psi(q)
	=
	\begin{cases}
		\psi_1(q)+\min\{0,q_1\}+\min\{0,q_2\} & \text{over} \ U_1 \\
		\psi_2(q)+\min\{0,q_2\}  	& \text{over} \ U_2 \\
		\psi_3(q) 							& \text{over} \ U_3
	\end{cases}
\end{equation}
It extends to a continuous map on $B=\mathbb R^3$.
For latter uses, we also set
\begin{equation}
		\label{psi_12_eq}
	\begin{aligned}
\psi^{(1)}&=\psi^{(1)} (q_1,q_2)=\psi(q_1,q_2, \log c_1)  \\
\psi^{(2)}&=\psi^{(2)} (q_1,q_2)=\psi(q_1,q_2,\log c_2)
\end{aligned}
\end{equation}
As in \cite{Yuan_local_SYZ}, it is standard to check that for any fixed $q_1, q_2$,
\begin{equation}
	\label{monotone_function_eq}
q_3\mapsto \psi(q)=\psi(q_1,q_2,q_3)
\end{equation}
is an increasing diffeomorphism from $\mathbb R$ to $(0,+\infty)$.
In particular, since $c_1>c_2>0$, we have $\psi^{(1)}>\psi^{(2)}>0$.

\subsection{Topological embedding map $j$}
\label{ss_j}
To visualize the induced integral affine structure on $B=\mathbb R^3$ (with singularities), we define a topological embedding
\begin{equation}
	\label{j_eq}
	j: B=\mathbb R^3\to\mathbb R^5 \qquad q=(q_1,q_2,q_3) \mapsto (\theta_1(q), \theta_2(q) , \vartheta(q),  q_1, q_2)
\end{equation}
where
\begin{equation}
	\label{theta_eq}
	\begin{aligned}
		\theta_1(q)&=\min\{ -\psi(q), -\psi^{(1)}(q_1,q_2)\} +\min\{0,q_1\}+\min\{0,q_2\} \\
		\theta_2(q)&=\min\{ \ \ \ \psi(q), \ \ \ \psi^{(2)}(q_1,q_2) \} \\
		\vartheta(q)&=   \median\{ \ \ \ \psi , \ \ \ \psi^{(1)}(q_1,q_2), \ \ \  \psi^{(2)}(q_1,q_2) \}
	\end{aligned}
\end{equation}
where $\median\{a,b,c\}$ denotes the median of three real numbers $a,b,c$.
Note that the above formulas implicitly rely on the condition $c_1>c_2>0$; in other words, if we had $c_1<c_2$, the formulas should be changed accordingly.

\begin{rmk}
	As we mentioned in Remark \ref{j_minor_role}, the role of $j$ is relatively minor. We recommend skipping this section about $j$ during the first reading and proceeding to the later sections discussing the tropically continuous fibration $F$. To some extent, the above formulas for $j$ are inspired by the computations surrounding $F$. Once we understand how the image of $F$ looks, discovering the appropriate formula for $j$ should become clearer.
\end{rmk}

Now, we use the various slices of $j$ fixing $q_1$ and $q_2$ to describe the image of $j$.
We consider the following map
\begin{equation}
	\label{r_eq}
r_{q_1,q_2}:   \mathbb R \to \mathbb R^3
\end{equation}
defined by
\[
t\mapsto 
\Big(
	\min\{-t, -\psi^{(1)} \}   + m ,\quad
	\min\{t, \psi^{(2)} \}  ,\quad
	\median\{t , \psi^{(1)} ,\psi^{(2)} \}
\Big)
\]
where we temporarily write 
\begin{equation}
	\label{m_eq}
	m:=m(q_1,q_2):= \min\{0,q_1\}+\min\{0,q_2\}
\end{equation}
In the remaining of this section, we shrink the domain of $r_{q_1,q_2}$ from $\mathbb R$ to $(0,+\infty)$ in view of (\ref{monotone_function_eq}).

The image of $r_{q_1,q_2}$ is an open broken line in $\mathbb R^3$ consisting of three linear components
\begin{equation}
	\label{r_cases_formula_eq}
R_{q_1,q_2}:  \qquad r_{q_1,q_2}(t)=
\begin{cases}
	(-\psi^{(1)}+ m, \  \  t, \ \ \psi^{(2)})  & \text{if} \ 0<t < \psi^{(2)} <\psi^{(1)} \\
	(-\psi^{(1)}+ m, \  \  \psi^{(2)}, \ \  t)  & \text{if} \ 0<\psi^{(2)}  \le t\le \psi^{(1)} \\
	(- t +m, \ \ \psi^{(2)}, \ \ \psi^{(1)} ) & \text{if} \ 0<\psi^{(2)}<\psi^{(1)}<t
\end{cases}
\end{equation}
with the two corner points when $t=\psi^{(2)}, \psi^{(1)}$:
\begin{equation}
	\label{corner_point_eq}
	\begin{aligned}
A& :=A(q_1,q_2) :=(-\psi^{(1)}+m, \psi^{(2)},\psi^{(2)} )  \\
A' & :=A'(q_1,q_2) := (-\psi^{(1)}+m, \psi^{(2)} ,\psi^{(1)} )
\end{aligned}
\end{equation}

On the other hand, we recall that the singular locus in $B=\mathbb R^5$ is given by $\Delta=\Delta_1\cup \Delta_2$ in (\ref{Delta_singular_locus_eq}). 
Then, we have that
\begin{equation}
	\label{j(Delta)_eq}
	\begin{aligned}
		j(\Delta_1) &= \{(A'(0,q_2), 0, q_2) \mid q_2\in\mathbb R \} \\
		j(\Delta_2) &= \{(A(q_1,0), q_1,0) \mid q_1\in\mathbb R\}
	\end{aligned}
\end{equation}

\section{Family Floer T-duality construction}
\label{s_mirror_construction}

\subsection{General aspects of non-archimedean geometry}
\label{ss_NA_general}

The \textit{Novikov field}
\[
\textstyle
\Lambda=\mathbb C((T^{\mathbb R}))=\{ \sum_{i=0}^\infty a_i T^{\lambda_i} \mid a_i\in\mathbb C, \lambda_i \to \infty\}
\]
is a non-archimedean field with the canonical valuation map
\begin{equation}
	\label{val_eq}
	\mathsf v: \Lambda\to\mathbb R\cup\{+\infty\}
\end{equation}
Consider the \textit{tropicalization map}
\[
\trop: (\Lambda^*)^n\to\mathbb R^n, \qquad (y_i)\mapsto (\val(y_i))
\]
The $(\Lambda^*)^n$ should be more precisely regarded as the non-archimedean analytification of the algebraic torus $\Spec \Lambda[[x_1^\pm,\dots, x_n^\pm]]$, but for clarity, we often make this point implicit (cf. \cite{Yuan_local_SYZ}).
In general, every algebraic variety $Y$ over a non-archimedean field admits an analytification space $Y^{\mathrm{an}}$.

Following Kontsevich-Soibelman \cite{KSAffine}, we study the notion of affinoid torus fibration, an analog of the integrable system in the non-archimedean analytic setting. Let $\mathcal Y$ be an analytic space over $\Lambda$ of dimension $n$, and let $B$ be an $n$-dimensional topological manifold or a CW complex. Let $f:\mathcal Y\to B$ be a proper continuous map for the analytic topology and the topology on the base. We call a point $p\in B$ \textit{smooth} (or \textit{$f$-smooth}) if there is a neighborhood $U$ of $p$ in $B$ so that the fibration $f^{-1}(U)\to U$ is isomorphic to $\trop^{-1}(V)\to V$ for some $V\subset \mathbb R^n$ that covers a homeomorphism $U\cong V$.
We also call $f^{-1}(V)$ an \textit{affinoid tropical chart}.
All other points are called \textit{singular}.
If all points of $B$ are smooth, then we call $f$ an \textit{affinoid torus fibration} (cf. \cite{NA_nonarchimedean_SYZ}).
If not, we denote by $B_0$ the set of all smooth points and call it the \textit{smooth locus} of $f$ that admits a natural integral affine structure \cite[\S 4]{KSAffine}.


Let $\mathcal Y$ be a Berkovich analytic space over a non-archimedean field and let $\mathcal B$ be a topological manifold of dimension $m$.
Let $F:\mathcal Y\to  \mathcal B$ be a continuous map with respect to the Berkovich topology on $\mathcal Y$ and Euclidean topology on $\mathcal B$.

\begin{defn}
	\label{tropically_continuous_defn}
	We say $F$ is \textit{tropically continuous} if for any point $x$ in $\mathcal Y$, there exist non-zero rational functions $f_1,\dots, f_N$ on an analytic neighborhood $\mathcal U$ of $x$, and there exists a continuous map $\varphi:U\to\mathbb R^m$ on an open subset $U\subset [-\infty, +\infty]^N$ such that
	\[
	F|_{\mathcal U} = \varphi(\val(f_1),\dots, \val(f_N))
	\]
\end{defn}


\begin{rmk}
	Following the definition by Chambert-Loir and Ducros in their work \cite[(3.1.6)]{Formes_Chambert_2012} closely, the functions $f_1, \dots, f_N$ should be invertible analytic functions. However, we relax the requirement to only nonzero rational functions, aligning more with Kontsevich-Soibelman's work \cite[\S 4.1]{KSAffine}. Indeed, we need to further include a continuous non-smooth function $\varphi: U \to \mathbb{R}^m$. This is inevitable for the reduced K\"ahler structure exhibits non-smooth points when moving the phase $s$ of the moment map. This extra function $\varphi$ is presented in \cite{Formes_Chambert_2012} but not in \cite{KSAffine}.
\end{rmk}

\subsection{Family Floer mirror construction: a quick review}
\label{ss_family_Floer_text}

In this paper, we only state the consequences directly for clarity.
We refer to \cite{Yuan_local_SYZ} for a concise survey and to \cite{Yuan_I_FamilyFloer} for the full details of the family Floer mirror construction. We also refer to \S \ref{ss_outline_intro} for an outline of the story.

Let $(X,\omega)$ be a symplectic manifold of real dimension $2n$ which is closed or convex at infinity. Suppose there is a smooth proper Lagrangian torus fibration $\pi_0:   X_0 \to B_0$ in an open subset $X_0$ of $X$ over an $n$-dimensional base manifold $B_0$.
There exists an integral affine structure on $B_0$ by Arnold-Liouville theorem.
Note that the quantum correction of holomorphic disks bounded by $\pi_0$-fibers is of global nature and sweep not just in $X_0$ but in $X$. Thus, we think of the pair $(X,\pi_0)$.
Let $L_q=\pi_0^{-1}(q)$ denote the Lagrangian fiber over $q$. Let $U_\Lambda$ be the \textit{unitary Novikov group}.

Assume that every Lagrangian fiber has zero Maslov classes and is preserved by an anti-symplectic involution. Although we may further weaken the assumption, it is already enough for our purpose.

\begin{thm}\hspace{-0.2em}\emph{\cite{Yuan_I_FamilyFloer}}
	\label{Main_theorem_thm}
	\ \ 
	We can associate to the pair $(X,\pi_0)$ a triple
	$
	\mathbb X^\vee:=(X_0^\vee,W_0^\vee, \pi_0^\vee)
	$
	consisting of a non-archimedean analytic space $X_0^\vee$ over $\Lambda$, a global analytic function $W_0^\vee$, and a dual affinoid torus fibration $\pi_0^\vee: X_0^\vee\to B_0$
	such that the following properties hold:

	\begin{enumerate}[i)]
		\itemsep 0pt
		\item The analytic structure of $X_0^\vee$ is unique up to isomorphism.
		\item The integral affine structure on $B_0$ from $\pi_0^\vee$ coincides with the one from the fibration $\pi_0$.
		\item The set of closed points in $X_0^\vee$ coincides with the disjoint union
		\begin{equation}
			\label{union_mirror_intro}
			R^1\pi_{0*}(U_\Lambda)\equiv \bigcup_{q\in B_0} H^1(L_q; U_\Lambda)
		\end{equation}
		of the sets of local $U_\Lambda$-systems on the $\pi_0$-fibers,
		and the map $\pi_0^\vee$ sends every $H^1(L_q; U_\Lambda)$ to $q$.
	\end{enumerate}
\end{thm}

To sidestep technical intricacies, this paper essentially adopts a practical viewpoint, emphasizing that once the correct algorithm is decided, it can be employed directly to generate examples and applications. 
This does not imply an absence of Floer-theoretic basis, but our priority is to ensure that the main result of this paper is essentially comprehensible without any prior specialized knowledge. This allows a broader audience to engage with the content.
In light of this, our review here primarily focuses on the algorithm. However, we certainly welcome readers who wish to delve into the Floer-theoretic foundational details in \cite{Yuan_I_FamilyFloer}.
See also \cite{Yuan_c_1,Yuan_affinoid_coeff,Yuan_e.g._FamilyFloer}.

Specifically, the algorithm goes as follows. We cover the smooth locus $B_0$ by small contractible open subsets $U_i$'s where $i\in \mathcal  I$ runs over some index set $\mathcal I$. We choose a point $q_i$ near $U_i$; note that $q_i\notin U_i$ is allowed. 
For the integral affine structure on $B_0$ induced by the Lagrangian fibration $\pi_0$, we can pick a \textit{pointed} integral affine coordinate chart
\[
\chi_i :(U_i, q_i) \xrightarrow{\cong} (V_i, c_i) \subseteq \mathbb R^n
\]
with $\chi_i(q_i)=c_i$. The local aspect of the dual affinoid torus fibration $\pi_0^\vee$ is represented by the following identification
\[
\tau_i :  (\pi_0^\vee)^{-1} (U_i) \xrightarrow{\cong} \trop^{-1}(V_i - c_i)
\]
for the aforementioned tropicalization map.
Now, Theorem \ref{Main_theorem_thm} means a canonical local-to-global gluing process that combines these local structures into the mirror space:
\[
X_0^\vee = \bigcup_{i\in I} \trop^{-1}(V_i-c_i) / \sim
\]
where $\sim$ indicates the gluing relation exclusively determined by Maslov index 0 holomorphic disks bounded by Lagrangian fibers, detailed in \cite{Yuan_I_FamilyFloer}.

The gluing relation is generally implicit; however, in this paper, detailed knowledge of places of Maslov-0 holomorphic disks and the symmetry of the Lagrangian fibration employed allow for explicit characterization.

\subsection{Void wall-crossing and explicit description}
Let $B_1\subset B_0$ be a contractible open set.
Let $B_2=\{x\in B_0\mid \dist (x, B_1) < \epsilon\}$ be a slight thickening of $B_1$ in $B_0$. We assume it is contractible and $\epsilon>0$ is sufficiently small.
Then, we have:

\begin{prop}
	\label{void_wall_cross_prop}
	Let $\chi: B_2\xhookrightarrow{} \mathbb R^n$ be an integral affine coordinate chart. If each Lagrangian fiber $L_{q}$ for $q\in B_1$ bounds no non-constant Maslov-0 holomorphic disk, then there is an affinoid tropical chart $(\pi_0^\vee)^{-1} (B_2) \cong \trop^{-1}(\chi(B_2))$.
\end{prop}

\begin{proof}
	Firstly, as $B_2$ is contractible, we may identify a fixed pointed integral affine chart $\chi: (B_2, q_0) \to (V, c) \subset \mathbb R^n$ at some point $q_0 \in B_2$. Next, we cover $B_2$ with sufficiently small pointed integral affine coordinate charts $\chi_i: (U_i, q_i) \to (V_i, v_i)$ where $i \in \mathcal{I}$. We may assume all $q_i$'s are within $B_1$, with no Maslov-0 disks along any Lagrangian isotopy among the fibers over $q_i$'s inside $B_1$.
	As above, we have multiple affinoid tropical charts $\tau_i: (\pi_0^\vee)^{-1}(U_i) \cong \trop^{-1}(V_i - v_i)$. Due to the absence of Maslov-0 holomorphic disks along Lagrangian isotopies among the fibers over $q_i$'s, the gluing maps for $\sim$ take the simplest form of the translation $y_k \mapsto T^{c_k} y_k$, $k=1,2,\dots, n$, within $(\Lambda^*)^n$ where $c_k$'s are some real numbers. Ultimately, it allows us to form a single affinoid tropical chart by merging all these $\tau_i$'s.
\end{proof}

By Theorem \ref{Main_theorem_thm}, we denote the mirror triple associated to $(X', \pi_0)$ in (\ref{pi_0}) by $(X_0^\vee, W_0^\vee, \pi_0^\vee)$.
Replacing $X'$ by $X$, we remark that the mirror associated to $(X,\pi_0)$ is simply the same $(X_0^\vee,  \pi_0^\vee)$ with the vanishing superpotential.
For the integral affine charts $\chi_k$'s in \S \ref{ss_action_coordinates}, using Proposition \ref{void_wall_cross_prop} yields three affinoid tropical charts as follows:
\begin{equation}
	\label{tau_123_eq}
	\begin{aligned}
\tau_1 : (\pi_0^\vee)^{-1}(U_1) \to (\Lambda^*)^3 , \qquad
\mathbf y & \mapsto  \big(T^{q_1} \mathbf y^{\sigma_1}, T^{q_2} \mathbf y^{\sigma_2}, T^{\psi_1(q)} \mathbf y^{\partial {\pmb \beta}_{1++}} \big) 
\\
\tau_2 : (\pi_0^\vee)^{-1}(U_2) \to (\Lambda^*)^3 , \qquad
\mathbf y & \mapsto  \big(T^{q_1} \mathbf y^{\sigma_1}, T^{q_2} \mathbf y^{\sigma_2}, T^{\psi_2(q)} \mathbf y^{\partial {\pmb \beta}_{2+}} \big) \\
\tau_3 : (\pi_0^\vee)^{-1}(U_3) \to (\Lambda^*)^3 , \qquad
\mathbf y & \mapsto  \big(T^{q_1} \mathbf y^{\sigma_1}, T^{q_2} \mathbf y^{\sigma_2}, T^{\psi_3(q)} \mathbf y^{\partial {\pmb \beta}_{3}} \big)
\end{aligned}
\end{equation}
where we use the natural pairing map
\[
H^1(L_q; U_\Lambda) \times \pi_1(L_q) \to U_\Lambda  \qquad (\mathbf y, \alpha)\mapsto \mathbf y^{\alpha}
\]
Recall that $U_\Lambda=\{x\in\Lambda\mid |x|=1\}=\{x\in\Lambda\mid \val(x)=0\}$.
Then, one can verify that 
\[\trop (\tau_k(\mathbf y))= \big(\val(T^{q_1}), \val(T^{q_2}), \val(T^{\psi_k(q)}) \big) =(q_1,q_2,q_3)=\chi_k(q) =\chi_k(\pi_0^\vee(\mathbf y))
\]
Their images are the analytic open domains in $(\Lambda^*)^3$, denoted by:
\begin{equation}
\label{T_k_eq}
T_k:=\trop^{-1}(\chi_k(U_k))
\end{equation}
for $k=1,2,3$.
In other words, $T_k\cong (\pi_0^\vee)^{-1}(U_k)$ via $\tau_k$.
The counts of the Maslov-2 holomorphic disks are well-known (cf. \S \ref{ss_topological_disks}), from which one can check that the restrictions $\mathcal W_k$ of $W_0^\vee$ over $T_k$ for $k=1,2,3$ are given by
\begin{align*}
	\mathcal W_1 &= y_3(1+y_1)(1+y_2)		\\
	\mathcal W_2 &= y_3(1+y_2)	\\
	\mathcal W_3 &= y_3		
\end{align*}
Here we use $y_1,y_2,y_3$ to denote the natural coordinates in $(\Lambda^*)^3$. We may apply (\ref{sigma_1sigma_2_eq}) and (\ref{q1+q2_eq}) here.

Denote by $\Phi_{12}$ (resp. $\Phi_{23}$) the corresponding analytic transition map from the chart $T_1$ to $T_2$ (resp. from $T_2$ to $T_3$).
Then, the \textit{family Floer mirror analytic spac}e $X_0^\vee$ in our specific conifold example is simply the gluing of there three charts
\begin{equation}
	\label{T_123_glue}
X_0^\vee \equiv T_1\sqcup T_2\sqcup T_3 /\sim
\end{equation}
where $\sim$ denotes the gluing relation induced by $\Phi_{12}$ and $\Phi_{23}$.

Due to the $T^2$-symmetry (\ref{Ham_act_eq}), we have that $\Phi_{12}$ and $\Phi_{23}$ both keep the first two coordinates $y_1$ and $y_2$. Besides, the existence of the global superpotential $W_0^\vee$ implies that $\mathcal W_1=\mathcal W_2 \circ \Phi_{12}$ and $\mathcal W_2=\mathcal W_3\circ \Phi_{23}$, and we can finally deduce that
\begin{align*}
	\Phi_{12} (y_1,y_2,y_3) &=  (y_1,y_2, y_3(1+y_1) )     \\
	\Phi_{23} (y_1,y_2,y_3) &=  (y_1,y_2, y_3(1+y_2) )
\end{align*}
Note that the integral affine structure on $B_0$ is exhibited as we identify the canonical dual affinoid torus fibration $\pi_0^\vee$ as the composition of three local segments of the tropicalization maps $\trop: T_k\to \chi_k(U_k)$.

\begin{rmk}
	Here we may regard $y_i$'s as the right-hand side of (\ref{tau_123_eq}), but they depend on the charts $\chi_k$'s or $\tau_k$'s. Moreover, if we choose different basis other than $\{\sigma_1,\sigma_2,\partial\pmb\beta_k\}$, the above expressions of $\mathcal W_k$'s and $\Phi_{ij}$'s will be all different accordingly. However, up to isomorphism, they will be the same.
\end{rmk}


\section{Tropically continous fibration on the resolved conifold: B side}
\label{s_B_side}

As in (\ref{Y_mirror_alg_var_eq_intro}), we consider the analytification $Y^{\mathrm{an}}$ of the following algebraic variety
\begin{equation}
	\label{Y_mirror_alg_var_eq}
	Y= \left\{ (x_1, x_2, z, y_1,y_2)\in\Lambda^2\times \mathbb P_\Lambda \times (\Lambda^*)^2  \ \  \mid \ \  
	\begin{aligned}
		x_1 z &=1+y_1 \\
		x_2 &=(1+y_2) z
	\end{aligned}
	\right\}
\end{equation} 
From now on, abusing the notation, the analytification $Y^{\mathrm{an}}$ will be still denoted by $Y$.

\subsection{Tropically continuous map $F$}
\label{ss_F}

Recall the function $\psi$ in (\ref{psi_123_together_eq}).
Inspired by the construction in \cite{Yuan_local_SYZ} and after lots of trials, we come up with the following tropically continuous map (\S \ref{ss_NA_general}):
\begin{equation}
	\label{F_eq}
F : Y \to \mathbb R^5 \qquad (x_1,x_2, z, y_1,y_2)\mapsto (F_1, F_2, G,  \val(y_1),\val(y_2) )
\end{equation}
where
\begin{align*}
	F_1&= \min \left\{
	\val(x_1), -\psi(\val(y_1),\val(y_2), \log c_1) +\min\{0, \val(y_1) \} +\min\{0, \val(y_2) \} 
	\right\} \\[0.5em]
	F_2&= \min \left\{ 
	\val(x_2), \ \ \ \psi(\val(y_1),\val(y_2), \log c_2) 
	\right\} \\[0.5em]
	G &=\median \{ \val(z)+\min\{0,\val(y_2)\}, \quad \psi(\val(y_1),\val(y_2), \log c_1), \quad \psi(\val(y_1),\val(y_2), \log c_2) \} 
\end{align*}

\begin{rmk}
The above construction can be regarded as implementing an \textit{ad hoc} strategy by Kontsevich and Soibelman in \cite[\S 8]{KSAffine}.
It turns out that there is a natural embedding from the family Floer mirror $X_0^\vee=T_1\cup T_2\cup T_3/\sim$ in (\ref{T_123_glue}) into the total space $Y$ of $F$ as we will see in the next section.
While we apologize for the lack of motivation, we hope the precision of the computation will be seen as a worthwhile endeavor.
Specifically, based on Conjecture \ref{conjecture_SYZ} and the Lagrangian fibration (\ref{pi_intro_eq}), the goal is to devise a tropically continuous map $f$ from the resolved conifold $Y$ to $\mathbb R^3$ such that the smooth part of $f$ induces an integral affine structure on the smooth locus that corresponds exactly with the integral affine structure induced by $\pi$. This is challenging due to the involvement of singular $f$-fibers.
The strategy involves shifting from identifying a map $f: Y \to \mathbb R^3$ to developing a map $F: Y \to \mathbb R^N$, where $N$ is large enough to "unfold" the singularities. Although this method is currently ad hoc, there is optimism that future research will formalize a more canonical approach, potentially through a version of Hartog’s extension.
\end{rmk}

Although $\val(x_1), \val(x_2)$ can take $+\infty$ and $ \val(z)$ can take $\pm\infty$, the corresponding values of $F_1, F_2, G$ are all necessarily finite after taking the `min' or the `median'.
Next, we describe the image of $F$ in $\mathbb R^5$.
For any $q_1,q_2$, we define
\[
S_{q_1,q_2} :=\{ (u_1,u_2, v )\in\mathbb R^3 \mid (u_1,u_2, v, q_1,q_2) \in F(Y)\}
\]
In other words,
\begin{equation}
	\label{F(Y)_image_eq}
F(Y)= \bigcup_{q_1,q_2\in\mathbb R} S_{q_1,q_2} \times \{(q_1,q_2)\} \quad \subset \mathbb R^5
\end{equation}
Suppose 
\[
p=F(  w)=(u_1,u_2,v,q_1,q_2)\in F(Y)
\]
for some $  w=(x_1,x_2,z,y_1,y_2)\in Y$.
Note that $q_1=\val(y_1)$ and $q_2=\val(y_2)$.
Recall the notations in (\ref{psi_12_eq}):
\begin{align*}
\psi^{(1)}=\psi^{(1)}(q_1,q_2)&=\psi(q_1,q_2,\log c_1)  \\
\psi^{(2)}=\psi^{(2)}(q_1,q_2)&=\psi(q_1,q_2,\log c_2)
\end{align*}
Due to (\ref{monotone_function_eq}), we also recall that $\psi^{(2)}<\psi^{(1)}$ for all $q_1,q_2$.
Recall the notation in (\ref{m_eq}): 
\[
m:=m(q_1,q_2)=\min\{0,q_1\}+\min\{0,q_2\} 
\]
Recall the broken line $r_{q_1,q_2}$ in (\ref{r_eq}):
\[
r_{q_1,q_2}: \mathbb R\to \mathbb R^3 \qquad t\mapsto 
\Big(
\min\{-t, -\psi^{(1)} \}   + m ,\quad
\min\{t, \psi^{(2)} \}  ,\quad
\median\{t , \psi^{(1)} ,\psi^{(2)} \}
\Big)
\]


Next, we aim to compare $S_{q_1,q_2}$ with the image $r_{q_1,q_2}(\mathbb R)$.
To do so, we consider the following cases of $q_1\equiv \val(y_1)$ and $q_2\equiv \val(y_2)$:

\subsubsection{Case 1: $q_1\neq 0$ and $q_2\neq 0$ . }

Notice that 
\begin{align*}
	\val(x_1)+\val(z) &=\val(1+y_1)= \min\{0, q_1\} \\
	\val(x_2)-\val(z)  &=\val(1+y_2)=\min\{0,q_2\}
\end{align*} 
$\val(x_1), \val(x_2), \val(z)$ are all finite real numbers, so $x_1,x_2,z\in\Lambda^*$.
	Eliminating $\val(x_1)$ and $\val(z)$ yields that
	\[
	(u_1,u_2,v)= 
	\big(
	\min\{
	-\val(x_2), -\psi^{(1)}
	\}
	+
	m
	\, \, , \, \,
	\min\{
	\val(x_2), \psi^{(2)}\}
	\, \,  , \, \,
	\median \{\val(x_2), \psi^{(1)}, \psi^{(2)} \}
	\big)
	\]
There is no constraint on $\val(x_2)$, and we may think of the variable $t:=\val(x_2)$.
So, whenever $q_1\neq 0 \neq q_2$, the $S_{q_1,q_2}$ agrees precisely with $r_{q_1,q_2}(\mathbb R)$.
Additionally, we can check that $p=(u_1,u_2,v,q_1,q_2)$ is always an $F$-smooth point (see \S \ref{ss_NA_general}).
	In reality, observe that $r_{q_1,q_2}(t)=(u_1,u_2,v)$.
	We take $0<\epsilon \ll 1$ and a small neighborhood $V'$ of $(q_1,q_2)$ such that $q'_1\neq 0\neq q'_2$ for any $(q'_1,q'_2)\in V'$.
	Then, there is a neighborhood $U$ of $p$ in $F(Y)$ that is homeomorphic to $V:=(t-\epsilon, t+\epsilon)\times V'$ by taking $r_{q_1',q_2'}(t' )$ for all $(q_1',q_2')$ in $V'$ and $t' \in(t-\epsilon,t+\epsilon)$.
	Under this identification $U\cong V$, we can see that $F^{-1}(U)$ is isomorphic to $\trop^{-1}(V)$ by forgetting $x_1$ and $z$, namely, by $(x_1,x_2, z, y_1,y_2)\mapsto (x_2, y_1,y_2)$.

\subsubsection{Case 2: $q_1=0$ and $q_2\neq 0$ . }
Notice that
\begin{align*}
	\val(x_1)+\val(z) &=\val(1+y_1)\ge \min\{0, q_1\} =0 \\
	\val(x_2)-\val(z)  &=\val(1+y_2)=\min\{0,q_2\}
\end{align*} 
Then, $\val(x_2)<+\infty$ and $\val(z)>-\infty$, i.e. $x_2\in \Lambda^*$ and $z\in \Lambda$.
Note that $\val(z)+\min\{0,q_2\}=\val(x_2)$.

\begin{enumerate}[label=(2\alph*)]
	\item 	If $\val(x_2)<\psi^{(1)}$, then $\val(x_1)>-\psi^{(1)}+m$ and
    \[
	(u_1,u_2,v)= (-\psi^{(1)} +m, \min\{\val(x_2), \psi^{(2)}\} , \median \{\val(x_2), \psi^{(1)} ,\psi^{(2)} \} )
	\]
	As above, we may regard $t:=\val(x_2)$ as a variable such that $-\infty <t<\psi^{(1)}$. Then, we note that $r_{0,q_2}(t)=(u_1,u_2,v)$ by (\ref{r_eq}), and we obtain
	\[
	r_{0,q_2} (-\infty, \psi^{(1)}) \subset S_{0,q_2}
	\]
	Besides, we can check that $p=(u_1,u_2,v,0,q_2)$ is $F$-smooth. In reality, we take $0<\epsilon\ll 1$ and a small neighborhood $V'$ of $(0,q_2)$ such that for any $t' \in (t-\epsilon,t+\epsilon)$ and $(q'_1,q'_2)\in V'$, we have $t' < \psi^{(1)}(q'_1,q'_2)$ and $q'_2\neq 0$.
	Then, there is a neighborhood $U$ of $p$ in $F(Y)$ that is homeomorphic to $V:=(t-\epsilon,t+\epsilon)\times V'$ by taking the various $r_{q'_1,q'_2}(t' )$. Under this identification $U\cong V$, we can see that $F^{-1}(U)$ is isomorphic to $\trop^{-1}(V)$ by forgetting $x_1$ and $z$.

	\item If $\val(x_1)<-\psi^{(1)}+m$, then $\psi^{(1)}<\val(x_2)<+\infty$ and
	\[
	(u_1,u_2,v) = ( \val(x_1) ,  \psi^{(2)}  ,    \psi^{(1)}) 
	\]
	Regard $s:= \val(x_1)$ as a variable such that $s<-\psi^{(1)}+m$.
Then, $\psi^{(1)}<m-s$ and $r_{0,q_2}(m-s)= (u_1,u_2,v)$ by (\ref{r_cases_formula_eq}).
Therefore,
	\[
	r_{0,q_2}(\psi^{(1)},+\infty) \subset S_{0,q_2}
	\]
	On the other hand, we can similarly check $p$ is $F$-smooth.
	In reality, we take $0<\epsilon\ll 1$ and a small neighborhood $V'$ of $(0,q_2)$ such that $q'_2\neq 0$ and $s' <-\psi^{(1)}(q'_1,q'_2) +m(q'_1,q'_2)$ for any $s'  \in(s-\epsilon,s+\epsilon)$ and $(q'_1,q'_2)\in V'$.
	Then, there is a neighborhood $U$ of $p$ in $F(Y)$ that is homeomorphic to $V:=(s-\epsilon,s+\epsilon)\times V'$
	by
	identifying $(s',q'_1,q'_2)$ with $r_{q'_1,q'_2}(m(q'_1,q'_2)-s')\equiv (s', \psi^{(2)}(q'_1,q'_2), \psi^{(1)}(q'_1,q'_2))$.
	Under this identification $U\cong V$, we can show that $F^{-1}(U)$ is isomorphic to $\trop^{-1}(V)$ by forgetting $x_2$ and $z$.

	\item If both $\val(x_1)\ge -\psi^{(1)}+m$ and $\val(x_2)\ge \psi^{(1)}$, then
	\[
	(u_1,u_2,v)= (-\psi^{(1)} +m, \psi^{(2)} ,  \psi^{(1)}  ) 
	\]
	coincides with the corner point $A'(0,q_2)$ in (\ref{corner_point_eq}). One can also check that $p$ is not $F$-smooth. (For instance, one may argue that a neighborhood of $p$ first contains a point $p'$ in the case (2a) and a point $p''$ in the case (2b); then, one may use the fact that the set of smooth points must admit an integral affine structure \cite[\S 4]{KSAffine}.)

\end{enumerate}

According to the above three bullets, we conclude that $S_{0,q_2}$ agrees with $r_{0,q_2}(\mathbb R)$ as well, while the set of $F$-smooth points in $S_{0,q_2}$ is given by $r_{0,q_2} ( (-\infty, \psi^{(1)})\cup (\psi^{(1)},+\infty))$.

\subsubsection{Case 3: $q_1\neq 0$ and $q_2=0$ . }
It is very similar to Case 2 with only slight change.
Notice that
\begin{align*}
	\val(x_1)+\val(z) &=\val(1+y_1)= \min\{0, q_1\}    \\
	\val(x_2)-\val(z)  &=\val(1+y_2)\ge \min\{0,q_2\} =0
\end{align*} 
It follows that $\val(z)<+\infty$ and $\val(x_1)<+\infty$, namely, $x_1\in \Lambda^*$ and $z\in \mathbb P_\Lambda\setminus \{0\}=\Lambda^*\cup\{\infty\}$. Observe that $\val(z)+\min\{0,q_2\}=-\val(x_1)+m$.

\begin{enumerate}[label=(3\alph*)]
	\item  If $\val(x_2)<\psi^{(2)}$, then $\val(x_1)>-\psi^{(2)} +m >  -\psi^{(1)} + m$ and
	\[
	(u_1,u_2,v) = (-\psi^{(1)} +m, \val(x_2), \psi^{(2)} )
	\] 
	Regard $t:=\val(x_2)$ as a variable with $t<\psi^{(2)}$, and we see that
	\[
	r_{q_1,0}(-\infty, \psi^{(2)} ) \subset S_{q_1,0}
	\]
	Besides, we can check that $p$ is $F$-smooth. In reality, note that $r_{q_1,0}(t)=(u_1,u_2,v)$. We take $0<\epsilon \ll 1$ and a small neighborhood $V'$ of $(q_1,0)$ such that for any $t'\in (t-\epsilon, t+\epsilon)$ and $(q'_1,q'_2)\in V'$, we have $t'<\psi^{(2)}(q'_1,q'_2)$ and $q'_1\neq 0$. Then, there is a neighborhood $U$ of $p$ in $F(Y)$ that is homeomorphic to
	$V:=(t-\epsilon,t+\epsilon)\times V'$ by taking the various $r_{q'_1,q'_2}(t')$.
	Under this identification, we know that $F^{-1}(U)$ is isomorphic to $\trop^{-1}(V)$ by forgetting $x_1$ and $z$.

	\item If $\val(x_1)< -\psi^{(2)} +m$, then $\val(x_2)>\psi^{(2)}$ and
	\[
	(u_1,u_2,v)
	=
	(
	\min\{\val(x_1) , -\psi^{(1)}+m\} , \,  \psi^{(2)}  \  ,  \, \median \{m-\val(x_1), \psi^{(1)}, \psi^{(2)} \}
	)
	\]
	Regard $s:=\val(x_1)$ as a variable such that $s<-\psi^{(2)}+m$. Then, $\psi^{(2)}<m-s$ and $r_{q_1,0}(m-s)=(u_1,u_2,v)$ (cf. (\ref{r_eq})).
	Therefore,
	\[
	r_{q_1,0}(\psi^{(2)} ,+\infty)\subset S_{q_1,0}
	\]
	We can similarly check that $p$ is $F$-smooth.

	\item If both $\val(x_1)\ge -\psi^{(2)}+m$ and $\val(x_2)\ge \psi^{(2)}$, then
	\[
	(u_1,u_2,v)= ( -\psi^{(1)}+m, \psi^{(2)}, \psi^{(2)})
	\]
	coincides with the corner point $A(q_1,0)$ in (\ref{corner_point_eq}). One can also check that $p$ is not $F$-smooth.
\end{enumerate}

Due to the above three bullets, we deduce that $S_{q_1,0}$ agrees with $r_{q_1,0}(\mathbb R)$, and the set of $F$-smooth points in $S_{q_1,0}$ is given by $r_{q_1,0}((-\infty, \psi^{(2)})\cup (\psi^{(2)}, +\infty))$.

\subsubsection{Case 4: $q_1=q_2=0$ . }

This is the most special case. Note that $m=0$ and $-\val(x_1)\le \val(z)\le \val(x_2)$.

\begin{enumerate}[label=(4\alph*)]
	\item If $\val(x_2)<\psi^{(2)}$, then $\val(z) <\psi^{(2)}$ and $\val(x_1) > -\psi^{(2)}>-\psi^{(1)}$. Hence,
	$
	(u_1,u_2,v)=( -\psi^{(1)} +m ,\val(x_2), \psi^{(2)})
	$, and one can imitate (3a) to show that
	\[
	r_{0,0}(-\infty, \psi^{(2)})\subset S_{0,0}
	\]
	consisting of $F$-smooth points.

	\item If $\val(x_1)<-\psi^{(1)}$, then $\val(x_2)\ge \val(z) >\psi^{(1)}>\psi^{(2)} $ and $(u_1,u_2,v, q_1,q_2)=(\val(x_1), \psi^{(2)} , \psi^{(1)})$, and one can imitate (2b) to show that
	\[
	r_{0,0}(\psi^{(1)},+\infty) \subset S_{0,0}
	\]
	consisting of $F$-smooth points.

	\item If $\psi^{(2)}<\val(z)<\psi^{(1)}$, then $\val(x_1)> -\psi^{(1)}$ and $\val(x_2)>\psi^{(2)}$. Hence,
	\[
	(u_1,u_2,v)=(-\psi^{(1)}+m, \psi^{(2)}, \val(z) )
	\]
	By viewing $\val(z)$ as a variable, one can similarly see that
	\[
	r_{0,0}( \psi^{(2)}, \psi^{(1)}) \subset S_{0,0}
	\]
	It consists of $F$-smooth points as well by imitating the arguments in (2a) and (3b).

	\item If $\val(x_2)\ge \psi^{(2)}$, $\val(x_1)\ge -\psi^{(1)}$, and $\val(z)\ge \psi^{(1)}$, then $(u_1,u_2,v)=(-\psi^{(1)}+m,\psi^{(2)}, \psi^{(1)})$.

	\item If $\val(x_2)\ge \psi^{(2)}$, $\val(x_1)\ge -\psi^{(1)}$, and $\val(z)\le  \psi^{(2)}$, then $(u_1,u_2,v)=(-\psi^{(1)}+m,\psi^{(2)}, \psi^{(2)})$.
\end{enumerate}

\subsubsection{Conclusion . }
Putting the Case 1, 2, 3, 4 in the above together and recalling the embedding map $j$ in (\ref{j_eq}) and (\ref{j(Delta)_eq}), we have proved the following theorem:

\begin{thm}
	\label{F_affinoid_thm}
	The $F$ restricts to an affinoid torus fibration over $\mathfrak B_0=F(Y) \setminus \hat\Delta$, where the singular locus is
	$
	\hat\Delta=\{
	(A'(0,q_2), 0, q_2) \mid q_2 \in\mathbb R\}
	\cup
	\{
	(A(q_1,0),q_1,0) \mid q_1\in\mathbb R
	\}
	$.
	Moreover, $j(\Delta)=\hat\Delta$ and
	\[
	j(B)= \{ (u_1,u_2,v, q_1,q_2) \in F(Y) \mid u_2 > 0\}
	\]
	In other words, if we set 
	$
	\mathscr Y:=  \{ |x_2|<1 \} \equiv \{\val(x_2)>0\}
	$
	in $Y$,
	then 
	\[
	j(B)=F(\mathscr Y)
	\]
\end{thm}

By Theorem \ref{F_affinoid_thm}, we can define a tropically continuous map as follows:
\begin{equation}
	\label{f_eq}
	f=j^{-1}\circ F : \mathscr Y\to B
\end{equation}
Moreover, since the singular locus $\hat\Delta$ of $F$ is identified with $j(\Delta)$, the singular locus of $f$ is given by $\Delta$ in $B=\mathbb R^3$. In particular, $f_0=f|_{B_0}$ is an affinoid torus fibration.

\begin{rmk}
	The role of the homeomorphism $j$ is actually quite minor, as it is simply designed to match the bases of $\pi$ and $F$. It is entirely fine to work with the Lagrangian fibration $j \circ \pi$ (instead of $\pi$) and the tropically continuous fibration $F$ (instead of $f=j^{-1} \circ F$).
\end{rmk}

We remark that Kontsevich and Soibelman first discovered such kind of construction in \cite{KSAffine} for the focus-focus singularities in dimension 2.
The above $f$ in the current paper closely follows their idea, and we further realize a new type of singular locus in dimension 3, which is not included in the class of examples in our previous work \cite{Yuan_local_SYZ}.

\section{Proof of main theorem}
\label{s_T_duality_matching}

\subsection{Analytic embedding $g$}
\label{ss_g}

The tropically continuous map $F: \mathscr Y\to B$ defined in the previous section is indeed ad hoc, and we sincerely apologize for the lack of motivation. However, it turns out that its total space $\mathscr Y$ largely agrees with the family Floer mirror space $X_0 = T_1\cup T_2\cup T_3/\sim$ in (\ref{T_123_glue}) via an analytic embedding $g$ defined below.
Recall that $T_k=\trop^{-1}(\chi_k(U_k))$ is defined as in (\ref{T_k_eq}) and $U_1, U_2, U_3$ gives an atlas of integral affine structure on the base manifold $B_0$ of the Lagrangian fibration.
Imitating the construction in \cite{Yuan_local_SYZ}, we define the analytic maps
\begin{equation}
	\label{g_k_eq}
	g_k: T_k\to  \Lambda^2\times\mathbb P_\Lambda \times (\Lambda^*)^2
\end{equation}
for $k=1,2,3$ by the following formulas:
\[
\begin{matrix}
	g_1(y_1,y_2,y_3)= &	\Big( 
	&  \displaystyle \frac{1}{y_3} ,  
	& y_3(1+y_1)(1+y_2) ,
	& y_3(1+y_1) , 
	& y_1,  & y_2 & \Big) \\[1.5em]
	g_2(y_1,y_2,y_3)= 
	&\Big( 
	& \displaystyle \frac{1+y_1}{y_3} ,
	& y_3(1+y_2) ,
	& y_3, 
	& y_1 , & y_2 & \Big) \\[1.5em]
	g_3(y_1,y_2,y_3)= &\Big( 
	&   \displaystyle \frac{(1+y_1)(1+y_2)}{y_3} , 
	& 	y_3 ,
	& \displaystyle \frac{y_3}{1+y_2} , 
	& y_1 , & y_2 & \Big) 
	\displaystyle 
\end{matrix}
\]
We check that
\[
g_2\circ \Phi_{12}=g_1, \qquad g_3\circ \Phi_{23}=g_2
\]
Hence, with regard to the identification (\ref{T_123_glue}), this gives rise to an analytic embedding map
\[
g: X_0^\vee \to \Lambda^2\times \mathbb P_\Lambda \times (\Lambda^*)^2
\]

\begin{rmk}
	The image of $g$ is contained in $Y$.
	The three \textit{analytic} open subsets $g_1(T_1), g_2(T_2), g_3(T_3)$ are strictly contained in the three \textit{Zariski} open subsets $\mathcal T_1=\{x_1\neq 0\}, \mathcal T_2=\{0\neq z\neq \infty\}, T_3=\{x_2\neq 0\}$ (cf. \S \ref{sss_speculation_Chan}).
	Hence, one can formally extend the domains of $g_k$'s to be $(\Lambda^*)^3$ and extend the domains of $\Phi_{12}, \Phi_{23}$ accordingly.
\end{rmk}

\subsection{A commutative diagram}
Recall that $f=j^{-1}\circ F$ as in (\ref{f_eq}). To prove Theorem \ref{Main_theorem_SYZ_intro}, it remains to show the following result. 
Note that we already identified an atlas of integral affine structure on $B_0$ with three charts (cf. (\ref{tau_123_eq}) and (\ref{T_k_eq})). The atlas on $B_0$ natrually induces an atlas on $j(B_0)\subseteq F(\mathscr Y)$ via the embedding $j$.
Note also that by definition, we already know $g$ is an analytic morphism due to its explicit formulas. Thus, it suffices to check the commutative diagram in the set-theoretic level. 

\begin{thm}
	\label{key_fibration_preserving}
	$F\circ g=j\circ \pi_0^\vee$. Namely, the following diagram commutes
	\[
	\xymatrix{
	X_0^\vee \ar[rr]^g  \ar[d]^{\pi_0^\vee} & & \mathscr Y \ar[d]^F \\
	B_0\ar[rr]^j & & \mathbb R^5
}
	\]
\end{thm}

\begin{proof}
Fix $\mathbf y$ in $X_0^\vee$. Write $q=(q_1,q_2,q_3)=\pi_0^\vee(\mathbf y)$ and 
$
g(\mathbf y)=(x_1,x_2,z,y_1,y_2)
$.

\begin{enumerate}
\item  If $q\in U_1$, then $\mathbf y$ is identified with a point $y=(y_1,y_2,y_3)$ in $T_1$ such that $g(\mathbf y)=g_1(y)$. Moreover, $\val(y_1)=q_1$, $\val(y_2)=q_2$, and $\val(y_3)=\psi_1(q)=\psi(q)-\min\{0,q_1\}-\min\{0,q_2\}$. 
Then,
		\begin{enumerate}[label=\theenumi\alph*)]
		\item By direct computation,\[
		F_1(g(\mathbf y))=\min\{ -\psi_1(q),  -\psi(q_1,q_2, \log c_1) +\min\{0,q_1\}+\min\{0,q_2\}\} =\theta_1(q)
		\]
		\item 
		We aim to show that
		\[
		F_2(g(\mathbf y)) =  \min\{ \psi_1(q)+\val(1+y_1)+\val(1+y_2), \psi(q_1,q_2,\log c_2)\}
		\]
		agrees with $\theta_2(q)$ in (\ref{theta_eq}). In fact, $q\in U_1$ implies that $q_3$ is $>$ or $\approx$ $\log c_1$, thus, $q_3> \log c_2$. By (\ref{monotone_function_eq}), 
		$\theta_2(q)=\psi(q_1,q_2,\log c_2)$ and the minimum in $F_2(g(\mathbf y))$ above takes the second number since
		$\psi_1(q)+\val(1+y_1)+\val(1+y_2) \ge \psi_1(q_1,q_2, \log c_2)+\min\{0,q_1\}+\min\{0,q_2\} =\psi(q_1,q_2, \log c_2)$.
		\item We aim to show that
		\[
		G(g(\mathbf y))= \median \{\psi_1(q)+\val(1+y_1) + \min\{0,q_2\},\quad  \psi(q_1,q_2, \log c_1), \quad \psi(q_1,q_2,\log c_2) \} 
		\]
		agrees with $\vartheta(q)$ in (\ref{theta_eq}). In fact, if $q_1=0$, then $q\in U_1$ implies $q_3>\log c_1>\log c_2$. It follows from (\ref{psi_123_together_eq}, \ref{monotone_function_eq}) that $\vartheta(q)=\psi(q_1,q_2,\log c_1)$ and that the median takes the second number since $\psi_1(q)+\val(1+y_1) +\min\{0,q_2\} \ge \psi(q_1,q_2,\log c_1)$.
		If $q_1\neq 0$, then $\val(1+y_1)=\min\{0,q_1\}$. In either case, what we want holds.
		\end{enumerate}

\item If $q\in U_2$, then $\mathbf y$ is identified with a point $y=(y_1,y_2,y_3)$ in $T_2$ such that $g(\mathbf y)=g_2(y)$. Moreover, $\val(y_1)=q_1$, $\val(y_2)=q_2$, and $\val(y_3)=\psi_2(q)=\psi(q)-\min\{0,q_2\}$. Then,
		\begin{enumerate}[label=\theenumi\alph*)]
			\item  We aim to show
			\[
			F_1(g(\mathbf y)) =\min\{  -\psi_2(q)+\val(1+y_1), -\psi(q_1,q_2,\log c_1)+\min\{0,q_1\}+\min\{0,q_2\} \}
			\]
			agrees with $\theta_1(q)$ in (\ref{theta_eq}).
			In fact, if $q_1=0$, then $q\in U_2$ implies that $q_3<\log c_1$. It follows from (\ref{monotone_function_eq}) that the minimum takes the second number and $-\psi_2(q)+\val(1+y_1) \ge -\psi_2(q_1,q_2,\log c_1) +\min\{0,q_1\}\equiv -\psi(q_1,q_2,\log c_1) +\min\{0,q_1\} +\min\{0,q_2\}=\theta_1(q)$.
			If $q_1\neq 0$, then $\val(1+y_1)=\min\{0,q_1\}$, so what we want still holds.
			\item We aim to show
			\[
			F_2(g(\mathbf y))=\min\{\psi_2(q)+\val(1+y_2) , \psi(q_1,q_2,\log c_2)\}
			\]
			agrees with $\theta_2(q)$ in (\ref{theta_eq}).
			In fact, if $q_2 =0$, then $q\in U_2$ implies that $q_3>\log c_2$. It follows from (\ref{monotone_function_eq}) that the minimum takes the second number and $\psi_2(q)+\val(1+y_2)\ge \psi_2(q_1,q_2,\log c_2)+\min\{0,q_2\}=\psi(q_1,q_2,\log c_2)=\theta_2(q)$.
			If $q_2\neq 0$, then $\val(1+y_2)=\min\{0,q_2\}$, so what we want still holds.
			\item By direct computation,
			\[
			G(g(\mathbf y))=\median\{ \psi_2(q)+\min\{0,q_2\} ,\quad \psi(q_1,q_2,\log c_1), \quad \psi(q_1,q_2,\log c_2) \} =\vartheta(q)
			\]
		\end{enumerate}

\item If $q\in U_3$, then $\mathbf y$ is identified with a point $y=(y_1,y_2,y_3)$ in $T_3$ such that $g(\mathbf y)=g_3(y)$. Moreover, $\val(y_1)=q_1$, $\val(y_2)=q_2$, and $\val(y_3)=\psi_3(q)=\psi(q)$. Then,
		\begin{enumerate}[label=\theenumi\alph*)]
			\item  We aim to show that
			\[
			F_1(g(\mathbf y))=\min\{ -\psi(q)+\val(1+y_1)+\val(1+y_2), \ -\psi(q_1,q_2, \log c_1) +\min\{0,q_1\}+\min\{0,q_2\} \}
			\]
			agrees with $\theta_1(q)$ in (\ref{theta_eq}).
			In fact, if $q_2=0$, then $q\in U_3$ implies that $q_3<\log c_2<\log c_1$. It follows from (\ref{monotone_function_eq}) that the minimum takes the second number, so $F_1(g(\mathbf y))=\theta_1(q)$.
			\item  By direct computation,
			\[
			F_2(g(\mathbf y))=\min\{ \psi(q) , \psi(q_1,q_2,\log c_2) \} =\theta_2(q)
			\]
			\item  We aim to show that
			\[
			G( g(\mathbf y))=\median \{\psi(q)-\val(1+y_2)+\min\{0,q_2\}, \psi(q_1,q_2,\log c_1) ,\psi(q_1,q_2,\log c_2)\}
			\]
			agrees with $\vartheta(q)$ in (\ref{theta_eq}). In fact, if $q_2=0$, then $q\in U_3$ implies that $q_3<\log c_2<\log c_1$.
			It follows from (\ref{monotone_function_eq}) that $\psi(q)-\val(1+y_2)+\min\{0,q_2\}\le \psi(q_1,q_2, \log c_2)$, so the median takes the third number. Hence, $G(g(\mathbf y))=\vartheta(q)$.
		\end{enumerate}
\end{enumerate}
The proof of Theorem \ref{key_fibration_preserving} is now complete. 
\end{proof}

\subsection{Proof of Theorem \ref{Main_theorem_SYZ_intro}}

Notice that $g$ is an analytic embedding map. Recall that by Theorem \ref{F_affinoid_thm}, the $f_0=f|_{B_0}$ is proved to be an affinoid torus fibration.
By Theorem \ref{key_fibration_preserving}, the image $g(X_0^\vee)$ agrees with $f^{-1}(B_0)$.
To wit, $g$ intertwines the affinoid torus fibrations $\pi_0^\vee$ and $f_0$.
Thus, the integral affine structure induced by $f_0$ precisely matches the one induced by $\pi_0^\vee$, and the latter, in turn, corresponds exactly to the integral affine structure induced by $\pi_0$ through its canonical construction as seen in (\ref{T_123_glue}).
This verifies the conditions (ii) and (iii). Finally, by Theorem \ref{F_affinoid_thm} again, the condition (i) holds as well.

\bibliographystyle{abbrv}

\bibliography{mybib_conifold}

\end{document}